\documentclass[12pt]{amsart}
\usepackage{amsfonts}
\usepackage{fullpage}
\usepackage{amsmath,amsthm}
\usepackage{mathrsfs}
\usepackage{latexsym}
\usepackage{array}
\usepackage{dsfont}
\usepackage{amssymb}
\usepackage{graphics}
\usepackage{enumerate}
\usepackage[english]{babel}
\usepackage[applemac]{inputenc}

\usepackage[T1]{fontenc}
\usepackage{bm} 
\usepackage{color}
\definecolor{marin}{rgb}   {0.,   0.3,   0.7} 
\definecolor{rouge}{rgb}   {0.8,   0.,   0.} 
\definecolor{sepia}{rgb}   {0.8,   0.5,   0.} 
\usepackage[colorlinks,citecolor=marin,linkcolor=rouge,
            bookmarksopen,
            bookmarksnumbered
           ]{hyperref}

\newtheorem{lemma}{Lemma}[section]

\newtheorem{theorem}[lemma]{Theorem}
\newtheorem{proposition}[lemma]{Proposition}

\newtheorem{remark}[lemma]{Remark}
\newtheorem{example}[lemma]{Example}

\newtheorem{notation}[lemma]{Notation}
\newtheorem{definition}[lemma]{Definition}
\newtheorem{conclusion}[lemma]{Conclusion}

\numberwithin{equation}{section}
\newcommand{\QED}{\mbox{}\hfill \raisebox{-0.2pt}{\rule{5.6pt}{6pt}\rule{0pt}{0pt}} 
          \medskip\par}

\newcommand{\eps}{\varepsilon}

\newcommand{\dd}{\mathrm{d}}

\newcommand{\Hc}{\mathcal{H}}

\newcommand{\N}{\mathbb{N}}
\newcommand{\Nc}{\mathcal{N}}

\newcommand{\R}{\mathbb{R}}

\newcommand{\C}{\mathbb{C}}

\newcommand{\Z}{\mathbb{Z}}

\newcommand{\Norm}[2]{\|#1\|\left.\vphantom{T_{j_0}^0}\!\!\right._{#2}}

\newcommand{\deltax}{{\delta x}}
\newcommand{\deltat}{h}
\renewcommand{\Re}{\mathrm{Re}}
\renewcommand{\Im}{\mathrm{Im}}
\newcommand{\ad}{\mathrm{ad}}
\author{}
\address{ } 
\email{}

\author{Erwan Faou \and Georg Maierhofer \and Katharina Schratz}
\address{ }
  \email{}

\title[]
{ Fully discrete backward error analysis for the midpoint rule applied to the nonlinear Schr\"odinger equation}

\begin{document}

\begin{abstract}  The use of symplectic numerical schemes on Hamiltonian systems is widely known to lead to favorable long-time behaviour. While this phenomenon is thoroughly understood in the context of finite-dimensional Hamiltonian systems, much less is known in the context of Hamiltonian PDEs. In this work we provide the first dimension-independent backward error analysis for a Runge--Kutta-type method, the midpoint rule, which shows the existence of a modified energy for this method when applied to nonlinear Schr\"odinger equations regardless of the level of spatial discretisation. We use this to establish long-time stability of the numerical flow for the midpoint rule.
\end{abstract}

\thanks{All authors gratefully acknowledge funding from the European Research Council (ERC) under the European
Union's Horizon 2020 research and innovation programme (grant agreement No.\ 850941). GM additionally gratefully acknowledges funding from the European Union's Horizon Europe research and innovation programme under the Marie Sklodowska--Curie grant agreement No.\ 101064261. EF is supported by the Simons collaboration on wave turbulence, the ANR project KEN ANR-22-CE40-0016, and benefited from the support of the Centre Henri Lebesgue ANR-11-LABX-0020-0. 
}

\maketitle
\section{Introduction}
In this work we study the symplectic midpoint rule applied to the nonlinear Schr\"odinger equation (NLSE)
\begin{align}\label{eqn:NLSE}
    \begin{cases}
        \partial_t z= -i\Delta z+i\lambda |z|^{2r}z,&(t,x)\in (0,T)\times \mathbb{R},\\
        z(0,x)=z_0,&x\in\mathbb{R},
    \end{cases}
\end{align}
where $r\in\mathbb{N}_{\geq 1}$ denotes the degree of the nonlinearity, and $\lambda \in \{\pm 1\}$ determines if the equation is {\em focusing} ($\lambda = -1$) or {\em defocusing} ($\lambda = 1$). We will consider a spatial discretization of this equation by finite differences and consider the family of fully discrete schemes depending on the time and space discretization parameters. 

The equation \eqref{eqn:NLSE} is a Hamiltonian partial differential equation (PDE) associated with the real energy 
\begin{equation}
\label{eq:Hcal}
\mathcal{H}(u,\bar u) = \int_{\R} |\nabla u(x)|^2 \dd x + \frac{\lambda}{r+1} \int_{\R} |u(x)|^{2r+2} \dd x
\end{equation}
which is preserved for all times
along smooth solutions of \eqref{eqn:NLSE}, and we can write this latter equation under the symplectic form $\partial_t z =  i \frac{\partial \Hc}{\partial \bar u}(z,\bar z)$ (cf. \cite[Section~III.1]{faou2012geometric} and \cite[Section~3.2]{marsden2002introduction}). Note that this equation also preserves the $L^2$ norm
\begin{equation}
\label{eq:L2}
\Nc(u) = \int_\R |u(x)|^2 \dd x. 
\end{equation}

In the case of finite-dimensional Hamiltonian systems the existence of a modified energy corresponding to the midpoint rule and, more generally, symplectic Runge--Kutta methods, is well-known since the work by Benettin \& Giorgilli \cite{benettin1994hamiltonian}, Murua \cite{murua1995metodos} and Tang \cite{tang1994} (cf. also \cite[Chapter IX.3]{hairer2013geometric}). This means, in the finite-dimensional case, that the discrete values given by the midpoint rule correspond to the evaluation of a continuous function which is the solution of a modified Hamiltonian system. This result is one of the central underpinnings of advantageous properties of symplectic integrators and permits a rigorous understanding of their long-time behaviour. Perhaps somewhat surprisingly results of this form (i.e. the existence of a modified energy and control on the long-time behaviour) are much more limited for symplectic integrators applied to partial differential equations (i.e. infinite-dimensional Hamiltonian systems). While in practise symplectic methods often exhibit good long-time behaviour \cite{faou2012geometric,celledoni2008symmetric,maierhofer_schratz_24} the aforementioned results for finite-dimensional systems do not translate easily to the infinite-dimensional case, essentially because the presence of unbounded operators means that analytic bounds derived for finite-dimensional cases break down when the spatial discretisation is refined. Recent work has provided some initial results resolving this problem by proving the existence of a modified energy for splitting methods for example in the work of Faou \& Gr\'ebert \cite{faou2011hamiltonian} and Bambusi et al. \cite{bambusi2013existence}.

In the present work we provide, for the very first time for a Runge--Kutta method, dimension-independent guarantees of the existence of a modified Hamiltonian applied to a discretisation of the NLSE \eqref{eqn:NLSE}. This is achieved by formulating the midpoint rule as a modified implicit-explicit splitting method involving pseudo-differential flows, and thus follows a two-step process: (i) firstly the existence of a suitable modified vector field is shown which leads to the splitting formulation of the midpoint rule; (ii) we use the implicit-explicit splitting decomposition and an approach based on \cite{faou2011hamiltonian,bambusi2013existence} to prove the existence of a modified energy for the full midpoint method.

An important point to notice is that {\em numerical resonances can a priori occur}, and that the existence of the modified energy requires the use of a CFL (Courant-Friedrichs-Lewy, \cite{CFL28}) restriction between the temporal and spatial discretisation parameters. This requirement is not surprising as it also appears in the context of splitting methods for the nonlinear Schr\"odinger equation.

The remainder of this manuscript is structured as follows. In Section~\ref{sec:problem_setting} we introduce the fully discrete NLSE which we consider for the remainder of this work, as well as useful notation for the presentation of later results. In Section~\ref{sec:midpoint_rule_as_splitting_method}, we then formulate the midpoint rule as an implicit-explicit (IMEX) splitting in the spirit of \cite{AscherReich99,SternGrinspun09,Zhang1997CheapIS}, see Propositions \ref{prop:splitting_formulation_on_flows} and \ref{prop:modified_polynomial_hamiltonian_expansion}. 
This is followed in Section~\ref{sec:modified_energy_for_midpoint_rule} by the formal construction and statement of our main result, which is given by Theorem~\ref{thm:modified_energy_midpoint} and which gives the existence of a modified energy under a CFL \eqref{eqn:CFL_condition_modified_energy} similar to the one used in \cite{faou2011hamiltonian}. Finally, as an application, in Section~\ref{sec:application_to_long-time_stability}, we prove the almost-global stability of the numerical scheme for small initial data in the energy space, see Theorem \ref{Theo:global}. 

\section{Problem setting and notation}\label{sec:problem_setting}
\subsection{The discrete NLSE} Let us first describe the spatial discretisation which we apply to the NLSE \eqref{eqn:NLSE} for the purpose of our analysis. We begin by approximating $\Delta f(x)\approx (\deltax)^{-2}(f(x+\deltax)-2f(x)+f(x-\deltax))$ which, together with a Dirichlet cut-off at $x=\pm (K+1)\deltax, K\in\mathbb{N}$ leads to the following system of ODEs called the discrete NLSE (see \cite{kevrekidis2009discrete} for the derivation, applications and references about this model) 
\begin{align}\begin{split}\label{eqn:discrete_NLSE}
    \begin{cases}
    \frac{\dd u_\ell}{\dd t}= i\frac{1}{(\deltax)^2}\left(-u_{\ell+1}+ 2u_\ell-u_{\ell-1}\right)+i\lambda |u_\ell|^{2r}u_\ell,&-K\leq \ell\leq K,\\
    u_{\pm(K+1)}=0,\\
    u_{\ell}(0)=z_{0}(\ell\deltax), &-K\leq \ell\leq K, 
    \end{cases}
    \end{split}
\end{align}
and we expect $u_\ell(t)$ to be an approximation of $z(t,\ell \delta x)$ the exact solution of \eqref{eqn:NLSE} at the grid points $\ell \delta x$. 
In the present work we are interested in studying this \textit{family} of spatially discrete problems for approximating the NLSE \eqref{eqn:NLSE} for arbitrary values $K,\deltax$. Note that in practice, fixing the length $X := K \delta x$ results in the Dirichlet problem for \eqref{eqn:NLSE}, i.e. transforms the problem from the unbounded domain $x \in \R$ to the same operator with Dirichlet boundary conditions on $[-X,X]$. We will not study the effect of this spatial truncation here (see \cite{bambusi2013existence,BernierFaou2019} for qualitative estimates in the case of solitons). 

For any fixed value fo $K,\deltax$ the previous system corresponds to the following ODE system in $2K+1$ dimensions
\begin{align}\label{eqn:model}
\frac{\dd u}{\dd t} = iAu + if(u),
\end{align}
with $f(u)_\ell= \lambda|u_\ell|^{2r}u_\ell$ and
\begin{align*}
    A=\frac{1}{\deltax^2}\begin{pmatrix}
        2&-1&\\
        -1&2&-1\\
        &-1&2&-1\\
        &&\ddots&\ddots&\ddots\\
        &&&-1&2&-1\\
        &&&&-1&2
    \end{pmatrix}.
\end{align*}
This equation turns out to be a Hamiltonian system associated with the energy (note here we think of $u$ as a column vector)
\begin{align}
\label{energy2}
H(u,\bar u) &= \delta x (\bar u^T A u) + F(u)\\
\nonumber &:= 2 \delta x \sum_{\ell = -K}^{K-1}  \frac{|u_{\ell+1} - u_{\ell}|^2 }{\delta x^2} 
+ \frac{\lambda}{r+1} \delta x \sum_{\ell = -K}^{K} |u_\ell|^{2r +2}, 
\end{align}
which is a discrete approximation of the continuous energy $\Hc$.
Note, moreover, that we can check directly that the system \eqref{eqn:discrete_NLSE} preserves the discrete $L^2$ norm
\begin{equation}
\label{eq:discrL2}
N(u) = \delta x \sum_{\ell = -K}^K |u_\ell|^2. 
\end{equation}
In particular, this preservation property ensures the global existence of the solution to the discrete NLSE system. 

\subsection{Midpoint rule}

To approximate the solution $u(t) = (u_\ell)_{\ell = -K}^K$ of \eqref{eqn:discrete_NLSE}, we discretize in time using a time step $h > 0$ and we consider the sequence $u^n = (u^n_\ell)_{\ell = -K}^K$ defined by induction  
$u^{n+1}:=\varphi_{h}(u^n)$ as the solution of the implicit equation
\begin{align}\label{eqn:midpoint_rule_nlse}
    u^{n+1}=u^{n} +\frac{i\deltat}{2}A(u^{n+1}+u^n)+i\deltat f\left(\frac{u^{n+1}+u^{n}}{2}\right). 
\end{align}
The sequence  $(u^n)_{n \geq \N}$ is then an approximation of the $u(n h)$ of \eqref{eqn:discrete_NLSE}. Note that this symplectic scheme preserves the $L^2$ norm $N(u^{n+1}) = N(u^n)$ which is a quadratic invariant of the problem (see \cite[Section VI.7]{hairer2013geometric}). 

Our aim is to establish uniform estimates for the existence of a modified Hamiltonian for \eqref{eqn:midpoint_rule_nlse}, meaning estimates which are valid for all systems in this family, i.e. independent of both $\deltax$ and $K$. We aim at proving the following result: For any given $N$, the midpoint rule coincide with the flow at time $h$ of a modified Hamiltonian system associated with a Hamiltonian function $H_h^{(N)}$ in the sense that 
$$
\varphi_h(u) = \Phi_{H_h^{(N)}}^h(u) + \mathcal{O}(h^{N+1}) \,\,\text{as}\,\, h\rightarrow 0,
$$
where $\Phi_P^t$ denote the flow of the Hamiltonian system of energy $P$. Such a result is not a surprise for general Hamiltonian system discretised with symplectic methods, \textit{when the spatial discretisation is fixed}. The main goal of this work is to make the previous construction and estimates {\em independent} of $K$ and $\delta x$. Such a result exist for splitting methods, see \cite{Debussche2009,faou2011hamiltonian,faou2012geometric} and can be used to prove stability results over long times for small solutions \cite{faou2012geometric}, solitary waves and plane waves \cite{bambusi2013existence,faou2014plane}. But this work is the first one concerning more classical symplectic Runge--Kutta methods, exemplified here by the midpoint rule. 

Before introducing the notation and the mathematical framework that will be used in this paper, let us remark that the linear case $\lambda = f = 0$ degenerates to the linear equation 
$$
u^{n+1} = R(hA) u^n = R(hA)^n u^0. 
$$
where $R$ is the stability function of the midpoint rule: 
$$
R(hA) = \frac{1 + i \frac{\deltat A}{2}}{1 - i\frac{\deltat A}{2}} = \exp\left( 2 i \arctan\left( \frac{\deltat A}{2}\right)\right). 
$$
This operator can be defined in several ways for example as 
\begin{align}\label{eqn:def_RdeltatA}
    R(\deltat A) = U^{-1}\exp\left( 2 i \arctan\left( \frac{\deltat D}{2}\right)\right)U,
\end{align}
where the action of the functions is understood to be on each element of the diagonal matrix $D$ and $U$ is a unitary matrix such that $A=U^{-1}DU$.
In this case, the backward error analysis is straightforwardly done: $u^n$ coincides with the solution at time $t = nh$ of the modified system 
$$
\frac{\dd}{\dd t} v = i \frac{2}{h}  \arctan\left( \frac{\deltat A}{2}\right) v. 
$$
ensuring the preservation of a modified energy for all times. This fact was used in \cite{Debussche2009,faou2011hamiltonian} to obtain long time energy estimates for splitting methods. 

\subsection{Functional setting}
The discrete space of functions is
$$
V_\deltax (= V_{\deltax,K}) =  \{u \in \C^{\mathbb{Z}}\vert u_{j}=0,|j|>K\}, 
$$
where the dependence in $K$ will remain implicit in the notation\footnote{This is consistent with practical applications, where we might take $K = X (\delta x)^{-1}$ with $X$ denoting the size of the Dirichlet cut-off, or the {\em large box} in which the problem on the real line is embedded. Note the case $X = 2\pi$ with periodic boundary conditions could be also tackled with a similar analysis.}. This space is 
equipped with the discrete norm 
$$
\Norm{\psi}{\deltax}^2 = 2\deltax \sum_{j\in\Z} \frac{|\psi_{j+1}-\psi_j |^2}{\deltax^2}+\deltax \sum_{j\in\Z} |\psi_j|^2,
$$
which is the norm associated with the real scalar product 
\begin{align}\label{eqn:bilinear_expression_discrete_norm}
    \langle \psi, \varphi \rangle_{\deltax}:=\deltax \,\mathrm{Re}\left[\overline{\psi}^T (I +A)\varphi\right], \qquad 
    \Norm{\psi}{\deltax} = \deltax \left[\overline{\psi}^T (I +A)\psi\right].
\end{align}
We can prove (see for instance \cite{bambusi2013existence}), that this norm is an algebra norm on $V_{\delta x}$, uniformly in $\delta x$, see Lemma~\ref{lem:V_h_is_algebra} below.

Following \cite{bambusi2009continuous}, we identify $V_\deltax$ with a finite element subspace of
$H^1(\R;\C)$. More precisely,  defining the function $s:\R \to \R$ by 
\begin{equation}\label{eqn:embedding_to_H1}
s(x)=
\begin{cases}
0\qquad\qquad  &{\rm if}\quad |x|>1,\\
x+1\quad  &{\rm if}\quad   -1\leq x\leq 0,\\
- x+1\quad  &{\rm if}\quad  0\leq x\leq 1,
\end{cases}
\end{equation}
 the identification is done through the map $i_\deltax: V_\deltax \to H^1(\R;\C)$ defined by 
\begin{equation}
\label{eq:ih}
\left\{\psi_j\right\}_{j\in\Z}\mapsto (i_\deltax\psi)(x) := \sum_{j\in\Z} \psi_j \ s\!\left(\frac{x}{h}-j\right) \ ,
\end{equation}
which we can easily check to be a continuous isomorphism between the two normed vector spaces {\em i.e.} there exists constant $c>0$ and $C$ independent of $\deltax$ and $K$ such that for all $v \in V_\deltax$, 
\begin{equation}
\label{eq:equiv}
c \Norm{v}{\deltax} \leq \Norm{i_{\deltax}(v)}{H^1} \leq C\Norm{v}{\deltax}. 
\end{equation}

\begin{lemma}\label{lem:V_h_is_algebra}
  $V_\deltax$ with the norm $\Norm{\cdot}{\deltax}$ is an algebra, with a constant independent of dimension. In particular, for any $v,w\in V_{\deltax}$,
  \begin{align*}
  \Norm{v\bullet w}{\deltax}\leq C\Norm{v}{\deltax}\Norm{w}{\deltax},
  \end{align*}
  where $C$ does not depend on $K$ and $\deltax$ and where $\bullet$ denotes the elementwise product of the two vectors. Occasionally we will drop the notation $\bullet$ when it is clear from context.
\end{lemma}
\begin{proof} We have
 \begin{align*}
     \Norm{v\bullet w}{\deltax}^2&=2\deltax \sum_{j\in\Z} \frac{|v_{j+1}w_{j+1}-v_jw_j |^2}{\deltax^2}+\deltax \sum_{j\in\Z} |v_j|^2|w_j|^2\\
     &\leq 2\deltax \sum_{j\in\Z} \frac{\left(|w_{j+1}||v_{j+1}-v_j|+|v_{j}||w_{j+1}-w_j|\right)^2}{\deltax^2}+\deltax \sum_{j\in\Z} |v_j|^2|w_j|^2\\
     &\leq 4\deltax \sum_{j\in\Z} \frac{|w_{j+1}|^2|v_{j+1}-v_j|^2+|v_{j}|^2|w_{j+1}-w_j|^2}{\deltax^2}+\deltax \sum_{j\in\Z} |v_j|^2|w_j|^2\\
     &\leq 2 \Norm{v}{\ell^\infty}^2\Norm{w}{\deltax}^2+2 \Norm{w}{\ell^\infty}^2\Norm{v}{\deltax}^2.
 \end{align*}
 Thus we have, using the definition \eqref{eq:ih} of $i_\deltax$, that there is a constant $C>0$ independent of $v,w,K$ such
 \begin{align*}
     \Norm{v\bullet w}{\deltax}\leq 2 \left(\Norm{i_\deltax(v)}{L^\infty(\mathbb{R})}\Norm{w}{\deltax}+\Norm{i_\deltax(w)}{L^\infty(\mathbb{R})}\Norm{v}{\deltax}\right).
 \end{align*}
 By Morrey's inequality it follows that for some $C>0$ independent of $v,w,K$ we have
 \begin{align*}
      \Norm{v\bullet w}{\deltax}\leq C\left(\Norm{i_\deltax(v)}{H^1(\mathbb{R})}\Norm{w}{\deltax}+\Norm{i_\deltax(w)}{H^1(\mathbb{R})}\Norm{v}{\deltax}\right).
 \end{align*}
 Thus the result follows by equivalence of the norms $v\mapsto \Norm{v}{\deltax}$ and $v\mapsto\Norm{i_{\deltax}(v)}{H^1(\mathbb{R})}$ on $V_{\deltax}$.
\end{proof}

\subsection{Hamiltonian formulation}
We use the following standard Hamiltonian coordinates $u=\frac{1}{\sqrt{2}}(p+iq)$. This translates to a coordinate-wise identification once we consider the values of $u_j$ at the node $j \in \Z$. 
With this complex representation we associate the derivatives 
$$
\frac{\partial}{\partial u_j} = \frac{1}{\sqrt{2}}\left(\frac{\partial} {\partial p_j} - i \frac{\partial} {\partial q_j} \right)\quad \mbox{and}
\quad 
\frac{\partial}{\partial \bar u_j} = \frac{1}{\sqrt{2}}\left(\frac{\partial} {\partial p_j} + i \frac{\partial} {\partial q_j} \right), \quad j \in \Z. 
$$
Any function $H(p,q)$ from $\R^{2K+1} \times \R^{2K+1} \to \R$ can be viewed as a function $H(u)$ defined on $\C^{2K+1}$ and taking real values\footnote{Note that with this identification, $H$ is in fact a function of $u$ and $\bar u$ and not a holomorphic function of $u$}. 
For such a Hamiltonian function $H:\mathbb{C}^{2K+1}\rightarrow \mathbb{R}$ we can then consider the vector 
\begin{align*}
    \nabla_{\overline{u}}H(u)\equiv \frac{1}{\sqrt{2}}\begin{pmatrix}\partial_p H\\\partial_q H\end{pmatrix}
\end{align*}
which allows us to introduce the vector field associated with a Hamiltonian function:
\begin{definition}[Hamiltonian formulation]
   The vector field, $X_H:\mathbb{C}^{2K+1}\rightarrow\mathbb{C}^{2K+1}$, associated with a Hamiltonian function $H$, is given by 
   \begin{align*}
       X_H(u):=i\deltax^{-1}\nabla_{\overline{u}}H,
   \end{align*}
   which in $(p,q)$-coordinates corresponds to
   \begin{align*}
       \deltax^{-1}\frac{1}{\sqrt{2}}J^{-1}\begin{pmatrix}\partial_p H\\\partial_q H\end{pmatrix},
       \quad J=\begin{pmatrix}0& I\\
    -I&0
    \end{pmatrix}\quad\mbox{with} \quad J^2 = - I,
   \end{align*}
and $I$ denoting the $(2K+1)\times (2K+1)$ identity. The Hamiltonian system associated with the function $H$ is 
\begin{align}\label{eqn:Hamiltonian_formulation_u_coords}
  \frac{\dd u}{\dd t}=X_H(u)=i\deltax^{-1}\nabla_{\overline{u}}H(u),
\end{align}
which is equivalent to the real system 
$$
\frac{d}{d t}
\begin{pmatrix}
p \\ q 
\end{pmatrix}
 = \delta x^{-1}J^{-1}\begin{pmatrix}\partial_p H\\\partial_q H\end{pmatrix} 
 = \delta x^{-1} \begin{pmatrix} -\partial_q H\\\partial_p H\end{pmatrix}.
$$
\end{definition}
\begin{remark}
    The scaling factor $\deltax^{-1}$ in the Hamiltonian formulation is included to make $\deltax^{-1}\nabla_{\bar{u}}$ consistent with a variational derivative in the limit $\deltax\rightarrow 0$. For example consider the functional $\Nc(u):=\int_{\mathbb{R}}|u(x)|^2\dd x$, whose discrete analogue in our setting \eqref{eqn:discrete_NLSE} is
    $N(u)=\deltax\sum_{j=-K}^K |u_j|^2$. The functional derivative of $\Nc$ is given by 
    \begin{align*}
    \delta_{\bar{u}}\Nc=u,
    \end{align*}
    while
    \begin{align*}
    \left(\nabla_{\bar{u}} N\right)_j=\frac{\partial N}{\partial \bar{u}_j}=(\delta x) u_j,
    \end{align*}
    ensuring that the correct scaling in the Hamiltonian formulation is indeed \eqref{eqn:Hamiltonian_formulation_u_coords}.
\end{remark}
\begin{definition}\label{def:symplectic_map}
 In this notation we then call a map $\Phi:\mathbb{C}^{2K+1}\rightarrow\mathbb{C}^{2K+1}$ symplectic if its Jacobian
\begin{align*}
    M=\begin{pmatrix}
        \frac{\partial \Re\Phi}{\partial p}& \frac{\partial \Re\Phi}{\partial q}\\
        \frac{\partial \Im\Phi}{\partial p}&\frac{\partial \Im\Phi}{\partial q}
    \end{pmatrix}
\end{align*}
satisfies
\begin{align}\label{eqn:def_symplecticity}
    M^TJM=J, \quad J=\begin{pmatrix}0& I\\
    -I&0
    \end{pmatrix}.
\end{align}
\end{definition}
\begin{remark}\label{rmk:symplectic_properties}
We can check (see \cite{hairer2013geometric}) that the midpoint rule \eqref{eqn:midpoint_rule_nlse} is symplectic in this sense (provided the solution is well defined as solution of an implicit system, see below). Moreover, the composition of symplectic maps is clearly symplectic.
\end{remark}
\begin{definition}[Commutator of vector fields]
For two real vector fields $X,Y$ from $\R^{2K+1} \times \R^{2K+1}$ to itself, we define the usual commutator $[X,Y]$ as follows
\begin{align*}
    [X,Y]=&\sum_{j=-K}^{K}\left(X_{j}^{(p)}\frac{\partial}{\partial_{p_j}}+X_{j}^{(q)}\frac{\partial}{\partial_{q_j}}\right)\left(Y_{j}^{(p)}\frac{\partial}{\partial_{p_j}}+Y_{j}^{(q)}\frac{\partial}{\partial_{q_j}}\right)\\
    &\quad\quad\quad\quad\quad-\left(Y_{j}^{(p)}\frac{\partial}{\partial_{p_j}}+Y_{j}^{(q)}\frac{\partial}{\partial_{q_j}}\right)\left(X_{j}^{(p)}\frac{\partial}{\partial_{p_j}}+X_{j}^{(q)}\frac{\partial}{\partial_{q_j}}\right),
\end{align*}
where $X_j^{(p)}$ and $X_j^{(q)}$ denote the $p_j$ and $q_j$ components of $X$ respectively. 

\end{definition}
\begin{definition}\label{def:Poisson_bracket}
The natural Poisson bracket in this formulation is given by
\begin{align*}
	\{F,G\}:=\deltax^{-1}\sum_{j=-K}^{K}\frac{\partial F}{\partial p_j}\frac{\partial K}{\partial q_j}-\frac{\partial G}{\partial p_j}\frac{\partial F}{\partial q_j}.
\end{align*}
\end{definition}
\begin{remark}
    The scaling in the Poisson bracket is important for consistency, as it ensures (as can be easily verified in the $(p,q)$-coordinates), that for two Hamiltonian functions $P,Q$,
    \begin{align*}
        [X_P,X_Q]=X_{\{P,Q\}}.
    \end{align*}
\end{remark}
\begin{remark}
We can also check that with the aforementioned scalings we have, as usual, that for any smooth function $g:\mathbb{C}^{2K+1}\rightarrow \mathbb{C}$ 
\begin{align*}
    \frac{\dd g(u)}{\dd t}=\{H,g\}.
\end{align*}
\end{remark}

Note we will in the following often switch between the formulation in $u\in \mathbb{C}^{2K+1}$ and in $(q,p)\in\mathbb{R}^{2K+1}\times\mathbb{R}^{2K+1}$. For this it is helpful to keep the identification $iu\equiv J^{-1}(q,p)$ in mind.

\subsection{Estimating polynomial vector fields} In order to establish the desired truncation bounds on the modified energy for the midpoint rule we have to introduce a suitable framework for estimating commutators of polynomial vector fields. For this we shall use the following notation introduced in \cite[Section 7.2]{bambusi2013existence}. Suppose $X$ is a vector field on $V_{\deltax}$ which is a homogeneous polynomial of degree $s$. Then we can associate (by polarization) with $X$ a symmetric mulilinear form $\tilde{X}(\psi_1,\dots,\psi_s)$ such that for all $\psi\in V_{\deltax}$, $X(\psi)=\tilde{X}(\psi,\dots,\psi)$. This symmetric multilinear form is given by 
\begin{align}\label{eqn:multilinear_form_tildeX}
    \tilde{X}(\psi_1,\dots,\psi_s) := \frac{1}{s!} \sum_{k=1}^s\sum_{1\leq j_1<\cdots<j_k\leq s}(-1)^{s-k}X(\psi_{j_1}+\cdots+\psi_{j_k}).
\end{align}
We can then define the following operator norm on such homogeneous polynomials
\begin{align}\label{eqn:norm_poly_vec_fields}
\Norm{X}{\deltax}=\sup_{\Norm{\psi_j}{\deltax}=1, j=1,\dots, s}\Norm{\tilde{X}(\psi_1,\dots,\psi_{s})}{\deltax} = \sup_{\Norm{\psi}{\deltax}=1}\Norm{X(\psi)}{\deltax},
\end{align}
see for instance \cite[Proposition 2]{Bochnak1971PolynomialsAM}. For notational convenience we introduce the following space.
\begin{definition}
    We denote by $\mathcal{P}_{s}$ the space of all polynomial vector fields $X$ of degree no larger than $s$ such that $\Norm{X}{\deltax}$ is uniformly bounded in both $K>0$ and $\deltax>0$.
\end{definition}
The norm \eqref{eqn:norm_poly_vec_fields} can then be extended to $\mathcal{P}_s$ by simply defining $\Norm{\,\cdot\,}{\deltax}$ of a general polynomial to be the sum of the norms applied to the homogeneous components. Note when $s=1$ the norm reduces to the usual operator norm on the normed vector space $V_{\deltax}$. Using \eqref{eqn:norm_poly_vec_fields} we can establish the following commutator estimate:
\begin{lemma}[{See Lemma 7.6 in \cite{bambusi2013existence}}]
\label{lem:estimates_multilinear_norm}
    Suppose $X\in\mathcal{P}_{s_1},Y\in\mathcal{P}_{s_2}$ are two polynomial vector fields. Then $[X,Y]\in \mathcal{P}_{s_1+s_2-1}$ and
\begin{align}\label{eqn:estimate_poisson_bracket}
    \Norm{[X,Y]}{\deltax}\leq (s_1+s_2)\Norm{X}{\deltax}\Norm{Y}{\deltax}.
\end{align}
\end{lemma}
\begin{proof}
    The proof of this statement is given in Lemma~7.6 \cite{bambusi2013existence}.
\end{proof}
\subsection{Vector fields and Hamiltonian functions}\label{sec:vector_fields_and_Hamiltonian_fns}
As usual there is a one-to-one correspondence between Hamiltonian functions and associated vector fields. We have already seen the definition of $X_H$ and in the following lemma we shall construct $H$ from $X$.
\begin{lemma}\label{lem:integration_symmetric_jacobian} Let $f:\mathbb{C}^{2K+1}\rightarrow\mathbb{C}^{2K+1}$ be a homogeneous polynomial vector field such that the matrix
\begin{align}\label{eqn:JtimesJacobian}
    J\begin{pmatrix}
        \frac{\partial \Re f}{\partial p}& \frac{\partial \Re f}{\partial q}\\[1ex]
        \frac{\partial \Im f}{\partial p}&\frac{\partial \Im f}{\partial q}
    \end{pmatrix}=\begin{pmatrix}
        \frac{\partial \Im f}{\partial p}&\frac{\partial \Im f}{\partial q}\\[1ex]
        -\frac{\partial \Re f}{\partial p}& -\frac{\partial \Re f}{\partial q}
    \end{pmatrix},
\end{align}
represented here in block-matrix notation, where $p=\Re(u)\in\mathbb{R}^{2K+1}, q=\Im(u)\in\mathbb{R}^{2K+1}$, is symmetric. Suppose further that, for given $C>0$ independent of $K$, $d\in\mathbb{N}_{\geq 1}$, $f$ satisfies the bound
\begin{align*}
    \Norm{f(u)}{\delta x}\leq C\Norm{u}{\deltax}^d,\quad\forall u\in V_{\delta x}.
\end{align*}
Then there is a \textit{real-valued} homogeneous polynomial $P:\mathbb{C}^{2K+1}\rightarrow\mathbb{R}$ such that for all $u\in\mathbb{C}^{2K+1}$
\begin{align}\label{eqn:hamiltonian_associated_with_f}
    f(u)=X_H=i\deltax^{-1}\nabla_{\overline{u}}P(u),
\end{align}
and there is $\tilde{C}>0$, independent of $K$, such that 
\begin{align}\label{eqn:bound_on_hamiltonian}
    |P(u)|\leq \tilde{C}\Norm{u}{\deltax}^{d+1},\quad \forall u \in V_{\delta x}.
\end{align}
\end{lemma}
\begin{remark}
    By slight abuse of notation we shall in the following write
    \begin{align}\label{eqn:notation_Jacobian}
        \nabla_{\bar{u}}^Tf\equiv\frac{1}{\sqrt{2}}\begin{pmatrix}
            \frac{\partial \Re f}{\partial p}& \frac{\partial \Re f}{\partial q}\\
        \frac{\partial \Im f}{\partial p}&\frac{\partial \Im f}{\partial q}
        \end{pmatrix}.
    \end{align}
\end{remark}
\begin{proof}[Proof of Lemma~\ref{lem:integration_symmetric_jacobian}] The proof is based on \cite[Lemma VI.2.7]{hairer2013geometric}. We define
\begin{align*}
    P(u):=2 {\deltax}\int_{0}^1\Im(\overline{u}^Tf(tu))\dd t,
\end{align*}
which clearly is a \textit{real-valued} homogeneous polynomial. Moreover, we have with $u = \frac{1}{\sqrt{2}}(p + iq)$  and using summation convention over repeated indices
\begin{align*}
    \left(i{\deltax^{-1}}\nabla_{\overline{u}}P\right)_{j}&=i\left(\frac{\partial}{\partial p_j}+i\frac{\partial}{\partial q_j}\right)\int_{0}^1 p_\ell\Im f_l(tu)-q_\ell\Re f_\ell(tu)\dd t\\
    &=\int_{0}^1\left(i\frac{\partial}{\partial p_j}-\frac{\partial}{\partial q_j}\right)\left(p_\ell\Im f_\ell(tu)-q_\ell\Re f_\ell(tu)\right)\dd t\\
    &=\int_0^1i\Im f_j(tu)+itp_\ell\frac{\partial \Im f_\ell}{\partial p_j}\Big\vert_{tu}-itq_\ell\frac{\partial \Re f_\ell}{\partial p_j}\Big\vert_{tu}\dd t\\
    &\quad+\int_0^1-tp_\ell\frac{\partial \Im f_\ell}{\partial q_j}\Big\vert_{tu}+\Re f_j(tu)+tq_\ell\frac{\partial\Re f_\ell}{\partial q_j}\Big\vert_{tu}\dd t.
\end{align*}
Using the symmetry of \eqref{eqn:JtimesJacobian} this simplifies to
\begin{align*}
    \left(i{\deltax^{-1}}\nabla_{\overline{u}}P\right)_{j}&=\int_0^1i\Im f_j(tu)+itq_\ell\frac{\partial \Im f_j}{\partial q_\ell}\Big\vert_{tu}+itp_\ell\frac{\partial \Im f_j}{\partial p_\ell}\Big\vert_{tu}\dd t\\
    &\quad+\int_0^1tq_\ell\frac{\partial \Re f_j}{\partial q_\ell}\Big\vert_{tu}+\Re f_j(tu)+tp_\ell\frac{\partial\Re f_j}{\partial p_\ell}\Big\vert_{tu}\dd t\\
    &=\int_0^1\frac{\dd}{\dd t}\left(t f_j(tu)\right)\dd t=f_j(tu).
\end{align*}
This completes the proof of \eqref{eqn:hamiltonian_associated_with_f}. For the bound \eqref{eqn:bound_on_hamiltonian} we note that for all $u\in B_{\deltax}(R)$
\begin{align*}
    |P(u)|&=2 \left|\int_{0}^1{\deltax}\Im(\overline{u}^Tf(tu))\dd t\right|\\
    &\leq 2 \sup_{t\in[0,1]}\left|{\deltax}\overline{u}^Tf(tu)\right|\leq 2 \Norm{\overline{u}}{\deltax}\sup_{t\in[0,1]}\Norm{f(tu)}{\deltax}\\
    &\leq 2C\Norm{\overline{u}}{\deltax}\sup_{t\in[0,1]}t^{d}\Norm{u^{d}}{\deltax},
\end{align*}
where we used the fact that for $u,v \in \C^{2K+1}$, 
$$
|\deltax \overline u^T v| \leq |\langle u,v\rangle_{\deltax}| \leq \Norm{u}{\deltax} \Norm{v}{\deltax}. 
$$ 
Thus the result follows from Lemma~\ref{lem:V_h_is_algebra}.
\end{proof}
The following stability estimate will also prove useful in Section~\ref{sec:analytic_estimates}.
\begin{lemma}\label{lem:stability_estimate_polynomial_vector_field}
    Suppose $P$ is a polynomial of degree $k$ such that $|P(u)|\leq C \Norm{u}{\deltax}^k$ then
    \begin{align*}
        \Norm{X_{P}(z)-X_{P}(y)}{\deltax}\leq C_kC\left(\max_{n=1,\dots, k-2}(\Norm{y}{\deltax}^n,\Norm{z}{\deltax}^n\right)\Norm{z-y}{\deltax}.
    \end{align*}
\end{lemma}
\begin{proof} It follows directly from the multilinear estimates \eqref{eqn:norm_poly_vec_fields} (see also Proposition 2.7 in \cite{faou2012geometric}). 
\end{proof}
\section{The midpoint rule as a splitting method}\label{sec:midpoint_rule_as_splitting_method}
Our goal is to resort to tools introduced in \cite{faou2011hamiltonian,faou2012geometric} to establish dimension-independent error estimates in the energy {(i.e. estimates which, beyond a CFL constraint, do not depend on $\deltax$ and $K$)}. In order for this to work, we need to first establish the following result (where we denote $B_{\deltax}(R):=\{u\in V_{\deltax}\,\vert\, \Norm{u}{\deltax}\leq R\}$):
\begin{proposition}\label{prop:splitting_formulation_on_flows}
Let $R>0$ be given. 
There exists $\deltat_0(R)>0$ and, for all $\deltat \leq \deltat_0$, a symplectic map $u \mapsto \Psi^{h}(u)$ from $B_{\deltax}(R)$ to $B_{\deltax}(2R)$ such that the midpoint rule can be written in the split form 
\begin{equation}
\label{eqsplit}
u^{n+1} = R(\deltat A) \circ \Psi^\deltat (u^n)
\end{equation}
and moreover, we have for $u \in B_{\deltax}(R)$ 
\begin{equation}
\label{eqsplit0}
\Psi^\deltat (u)= u + \sum_{k \geq 1} \deltat^k \Psi_{\deltat,k} (u)
\end{equation}
where the $\Psi_{\deltat,k}(u)$ are homogeneous polynomial vector fields such that 
there exists a constant $C$ such that for all $\deltat \leq \deltat_0$ and all $u$ in $B_{\deltax}(R)$
\[
\Norm{\Psi_{\deltat,k}(u)}{\deltax} \leq C^k \Norm{u}{\deltax}^{2rk + 1},
\]
where $C$ does not depend on $N=2K+1$, the dimension of the ODE, nor on $\deltax$ the mesh size of the space discretisation.
\end{proposition}
To prove the above splitting formulation, we need a few lemmas.
\begin{lemma}\label{lem:RA_isometry}
    $R(\deltat A)$ is an isometry on $\Norm{\cdot}{\deltax} $.
\end{lemma}
\begin{proof}
    The eigenvalues of $R(\deltat A)$ are all on $i\mathbb{R}$, and it can be jointly diagonalised with the norm $\Norm{\cdot}{\deltax} $, \em{cf}. \eqref{eqn:bilinear_expression_discrete_norm}.
\end{proof}

\begin{proof}[Proof of Proposition~\ref{prop:splitting_formulation_on_flows}]
	We can write the midpoint rule as 
	\begin{align*}
		u^{n+1} &= R(hA) u^{n} + \frac{i h}{1 - i hA/2} f\!\left(\frac{u^n + u^{n+1}}{2}\right) \\
		&= R(hA) \left( u^{n} + \frac{i h}{1 + i hA/2} f\!\left(\frac{u^n + u^{n+1}}{2}\right) \right) .
	\end{align*}
	Let 
	$$
	v^{n+1} = R(hA)^* u^{n+1}. 
	$$
	We have 
	\begin{align*}
		v^{n+1}=  u^{n} + \frac{i h}{1 + i hA/2} f\!\left(\frac{u^n + R(hA) v^{n+1}}{2}\right).
	\end{align*}
For $\varepsilon > 0$ and $u$ fixed, we define the map
	\begin{equation}
	\label{eq:defF}
	v \mapsto F_{\varepsilon,h,u}(v) = u + \frac{i\varepsilon }{1 + i hA/2} f\!\left(\frac{u + R(hA)v}{2}\right). 
	\end{equation}
    With $R(hA)$ being an isometry of $V_{\deltax}$, we have that
	$\frac12(u + R(hA)v) \in B_{\deltax}(\frac{3\|u\|_{\deltax}}{2})$ for any $v\in B_{\delta x}(2\|u\|_{\deltax})$ and thus $\Norm{\phi(\frac{u + R(hA)v}{2})}{\deltax} \leq C \Norm{u}{\deltax}^{2r +1}$ for any $v\in B_{\delta x}(2\|u\|_{\deltax})$ and for some constant $C>0$ depending only on $r$. Finally, we have 
	$$
	\left\|\frac{w}{1 + i hA/2}\right\|_{\deltax} \leq \Norm{w}{\deltax}
	$$
	similarly to the arguments of the proof of lemma~\ref{lem:RA_isometry} and, thus, we deduce that 
	$$
	\Norm{F_{\varepsilon,h,u}(v)}{\deltax} \leq \Norm{u}{\deltax} + C \varepsilon \Norm{u}{\deltax}^{2r+1} \leq 2\Norm{u}{\deltax}
	$$
	for $\varepsilon \leq \varepsilon_0=C^{-1}\|u\|_{\deltax}^{2r}$ if $v\in B_{\delta x}(2\|u\|_{\deltax})$. This shows that $F_{\varepsilon,\deltat,u}$ maps $B_\deltax(2 \|u\|_{\deltax})$ to itself (so long as $0<\varepsilon\leq \varepsilon_0$). Now we estimate for $v$ and $w$ in $B_{\deltax}(2 \|u\|_{\deltax})$:
	$$
	\Norm{F_{\varepsilon,h,u}(v) - F_{\varepsilon,h,u}(w)}{\deltax} \leq C \varepsilon   \Norm{u}{\deltax}^{2r} \Norm{v - w}{\deltax},
	$$
	where $C$ depends on $r$ but not on $K,\deltax$, and thus for $0<\varepsilon<\varepsilon_0$, $F_{\varepsilon,\deltat,u}$ is a contraction from $B_{\deltax}(2 \|u\|_{\deltax})$ to $B_{\deltax}(2 \|u\|_{\deltax})$ ensuring the existence of a fixed point $v \in B_{\deltax}(2\|u\|_{\deltax})$ such that $v = F_{\varepsilon,\deltat,u}(v)$ and we can define 
 \begin{align}\label{eqn:def_Psi_eps^h}
 \Psi_{\varepsilon}^{h}(u):=v.
 \end{align}
 Note that this argument also shows the well-posedness of the midpoint rule for $h$ small enough, depending on the size of the numerical solution. 
 
 Next we would like to show that $\Psi^\deltat_{\varepsilon} (u)$ has an expansion of the form
		\begin{align}\label{eqn:expansion_psi_deltat_varepsilon}
		\Psi^\deltat_{\varepsilon} (u)= \sum_{k \geq 0} \varepsilon^k \Psi_{\deltat,k} (u).
	\end{align}
For this we apply the implicit function theorem. We define 
\begin{align*}
g(\varepsilon,v):= v- F_{\varepsilon,h,u}(v),
\end{align*}
where 
$v$ is viewed as a vector in $\R^{2K+1} \times \R^{2K+1}$ through the identification $v = p + iq$. 
The map is entire in $(\varepsilon,p,q)$ (note that it is an entire function of $(\varepsilon, v,\bar v)$ 
and moreover $\Psi^\deltat_{\varepsilon}$ is the solution to $g(\varepsilon,\Psi^\deltat_{\varepsilon}(u) )=0$). Let us consider the Jacobian of $g$ around $v=u=:\Psi_{\deltat,0}(u)$ applied to a vector $z= r + i s\in\mathbb{C}^{2K+1} \equiv \R^{4K +2}$:
\begin{align*}
	dg_{\varepsilon,h,u}(u)(z)&=
  z-\frac{i\varepsilon \lambda ( r+1)}{2 + i hA}\left[\left|\frac{u + R(hA)v}{2}\right|^{2r}\bullet (R(hA)z)\right]\\
	&\quad\quad-\frac{i\lambda \varepsilon r}{2 + i hA}\left[\left|\frac{u + R(hA)v}{2}\right|^{2r-2} \bullet \left(\frac{u + R(hA)v}{2}\right)^2\bullet (R(hA)^*\bar z)\right],
\end{align*}
where $\bullet$ denotes element-wise multiplication of two vectors. Thus we have by the above lemmas and the fact that we already showed $\|\Psi^\deltat_{\varepsilon}(u)\|_{\deltax}\leq 2\|u\|_{\deltax}$,
\begin{align*}
	\|dg_{\varepsilon,h,u}(u)(z)-z\|_{\deltax}\leq C\varepsilon\Norm{u}{\deltax}^{2r}\Norm{z}{\deltax},
\end{align*}
for some constant $C$ independent of $K,\deltax$ and $u$. 
Therefore, so long as $0<\varepsilon<\varepsilon_0=\frac12 C^{-1}\|u\|_{\deltax}^{-{2r}}$, the Jacobian of the map $g(\varepsilon,v)$ is invertible and thus, by analyticity of $g$ and the implicit function theorem, the map $\varepsilon\mapsto\Psi^\deltat_{\varepsilon}(u)$ is analytic in the region $|\varepsilon|<\varepsilon_0$. Therefore, the expression \eqref{eqn:expansion_psi_deltat_varepsilon} indeed holds for a sequence of vector fields $\Psi_{\deltat,k} (u)\in V_{\deltax}$ which, by Cauchy's estimate satisfy
\begin{align}
\label{boundpsikh}
	\|\Psi_{\deltat,k} (u)\|_{\deltax}\leq \frac{\sup_{|\varepsilon|=\varepsilon_0}\|\Psi^\deltat_{\varepsilon} (u)\|_{\deltax}\varepsilon_0^{-k}}{2\pi}\leq C^k \|u\|_{\deltax}^{2k+1},
\end{align}
where we have increased $C$ appropriately without changing the notation in the interest of simplicity. Finally, to show that each $\Psi_{\deltat,k}$ is a homogeneous polynomial of degree $\leq 2k+1$ we proceed by induction on $k$ as follows: The statement is trivially true for $k=0$ since $\Psi_{\deltat,0}(u)=u$. Suppose we have shown the claim for $\Psi_{\deltat,j}$ with $0\leq j\leq k-1$. We consider the expansion
\begin{align}\label{eqn:truncated_expansion_Psi}
	\Psi^h_{\varepsilon}(u)=\sum_{j=0}^{k}\varepsilon^{j}	\Psi_{h,j}(u)+\varepsilon^{k+1}\mathcal{R}_{k}
\end{align}
then we have
\begin{align*}
	\Psi^h_{\varepsilon}(u)=F_{\varepsilon,h,u}(\Psi^h_{\varepsilon}(u)),
\end{align*}
and expanding the right hand side ($F$ is a composition of polynomials and linear operators) we find
\begin{align*}
	\Psi^h_{\varepsilon}(u)=u+\frac{i \lambda \varepsilon}{1+ihA/2}\sum_{j=0}^{k}\varepsilon^{j}\sum_{\substack{m_1+\cdots+m_{r}+n_1+\cdots+n_{r+1}=j\\0\leq m_1,\dots,m_{r},n_{1},\dots,n_{r+1}\leq j}}\overline{\alpha_{m_1}}\cdots \overline{\alpha_{m_r}}\alpha_{n_1}\cdots\alpha_{n_{r+1}}+\varepsilon^{k+2}\tilde{\mathcal{R}}_k,
\end{align*}
where
\begin{align*}
\alpha_j=\begin{cases}
	\frac{u+R(hA)u}{2},& j=0,\\
	\frac{R(hA)}{2}\Psi_{h,j}(u),& j\geq1,
\end{cases}
\end{align*}
and the map $u\mapsto\tilde{\mathcal{R}}_k(u)$ is bounded. Using \eqref{eqn:truncated_expansion_Psi} and comparing the coefficient on both sides immediately shows that each $\Phi_{h,\varepsilon,k}(u)$ is a homogeneous polynomial of degree $\leq 2k+1$ in $u$. In particular, this provides a recursive way of computing these expressions and the first two of them are
\begin{align}
\label{firsterms}
\Psi_{h,0}(u)=u,\quad \Psi_{h,1}(u)=\frac{i \lambda }{1+ihA/2}\left|\frac{u+R(hA)u}{2}\right|^{2r}\frac{u+R(hA)u}{2}.
\end{align}
Finally, we note that the above restrictions on $\varepsilon$ were completely independent of $h>0$, thus we can take $\varepsilon=h<\varepsilon_0$ and the result follows with $\Psi^h = \Psi_h^h$.
\end{proof}

Note that using \eqref{firsterms}, we can write the expansion in the form 
\begin{align*}
\Psi^h(u) &= u + \frac{i\lambda h}{1+ihA/2}\left| \frac{i}{1- ihA/2} u\right|^{2r} \frac{i}{1- ihA/2} u + \mathcal{O} (\Norm{u}{\deltax}^{2r + 1})\\
&= u + i h (\deltax)^{-1}\nabla_{\bar u} P_{0,h}(u) + \mathcal{O} (\Norm{u}{\deltax}^{2r + 1}),
\end{align*}
where 
\begin{equation}
\label{P0h}
P_{0,h}(u)= 
\frac{\lambda}{r+1} \delta x \sum_{\ell = -K}^{K} \left| \left( \frac1{1 - ih A/2} u\right)_\ell\right|^{2r +2}, 
\end{equation}
is the $(2r+2)$-part of the original energy (see \eqref{energy2}) composed with the pseudo-differential operator $(1 - ih A/2)^{-1}$. 
This suggests that $\psi^h$ is the Hamiltonian flow at time $h$ of a modified bounded pseudo-differential operator depending on $h$. 
For this we recall the general Hamiltonian formulation in the $u$ coordinate \eqref{eqn:Hamiltonian_formulation_u_coords} and exploit Lemma~\ref{lem:integration_symmetric_jacobian}.

\begin{proposition}\label{prop:modified_polynomial_hamiltonian_expansion}
	The map $u \mapsto \Psi^{h}_\varepsilon(u) $ as introduced in \eqref{eqn:def_Psi_eps^h} is symplectic for any permissible choice of $\varepsilon,h$ (in particular for $\varepsilon=h$, i.e. $\Psi^h$ is symplectic). Moreover, there exists a formal {real-valued} Hamiltonian 
	$$
	P_h = P_{0,h} + \sum_{k \geq 1} h^k P_{k,h},
	$$
	where the $P_{k,h}$ are \textit{real-valued} homogeneous polynomials with
 \begin{align*}
     |P_{k,h}(u)|\leq C^{(k)}\|u\|_{\deltax}^{2r(k+1)+2}
     \end{align*}
    with $C^{(k)}>0$ independent of $K,\deltax$, such that the following holds. For any given $N$ there exists $C_N,\tau^{(N)}_0>0$ such that if we define
	$$
	P_h^{(N)} = P_{0,h} + \sum_{k = 1}^N h^k P_{k,h}
	$$
	then we have for any $u\in B_{\deltax}(R), h\in [0,\tau_{0}^{(N)})$
	$$
	\Psi^{h}(u) = \Phi^h_{P_h^{(N)}}(u) + R_{h,N}(u),
	$$
	where $\Phi^t_{P_h^{(N)}}$ is the flow associated with the vector field\ \ $i\deltax^{-1}\nabla_{\overline{u}}P_h^{(N)}$ and the remainder terms $R_{h,N}$ take the form
 \begin{align}\label{eqn:induction_step}
     R_{h,N}(u)=\sum_{\ell=N+2}^\infty h^\ell\mathcal{R}^{(N)}_{h,\ell}(u),
 \end{align}
and each $\mathcal{R}^{(N)}_{h,\ell}$ is a polynomial vector field satisfying the bound
 \begin{align}\label{eqn:estimate_showing_analyticity_of_Phi_PN}
    \Norm{ \mathcal{R}^{(N)}_{h,\ell}(u)}{\deltax}\leq \tilde{C}^{(N)}\left(C^{(N)}\right)^{\ell}\Norm{u}{\delta x}^{2r\ell+1},
 \end{align}
for some $\tilde{C}^{(N)},C^{(N)}>0$ independent of $K,\deltax$.
\end{proposition}
\begin{lemma}[Cauchy-Kovalevskaya]\label{lem:flow_of_polynomial_hamiltonian} Suppose $\mathcal{P}:\mathbb{C}^{2K+1}\rightarrow\mathbb{R}$ is a real-valued homogeneous polynomial such that
\begin{align*}
    |\mathcal{P}(u)|\leq C\Norm{u}{\deltax}^d,\quad \forall u\in\mathbb{C}^{2K+1}
\end{align*}
for some constant $C>0$, and some $d\in\mathbb{N}_{\geq 2}$. Let us fix $R>0$. Then, there is a $\tau_0>0$ dependent only on $d, C$ such that for all $0\leq t<\tau_0$ the flow map $\Phi_{\mathcal{P}}^t:B_{\deltax}(R)\subset\mathbb{C}^{2K+1}\rightarrow\mathbb{C}^{2K+1}$ associated with the differential equation
\begin{align}\label{eqn:differential_equation_Cauchy_Kovalevskaya}
    \frac{\dd u}{\dd t}=i{(\deltax)^{-1}}\nabla_{\overline{u}} \mathcal{P}(u),
\end{align}
exists, is symplectic in the sense of Definition~\ref{def:symplectic_map}, and is locally analytic. In particular there is $\tilde{C}_0>0$ (again dependent only on $C,d$) such that for any $u\in B_{\deltax}(R)$ and $\varepsilon \leq \tau_0$ we have 
\begin{align}\label{eqn:expression_of_solution_Cauchy_Kovalevskaya}
\Phi_{\mathcal{P}}^\varepsilon(u)=u+i\varepsilon{(\deltax)^{-1}}\nabla_{\overline{u}} \mathcal{P}(u)+\sum_{k\geq 2}\varepsilon^k\mathcal{R}_{k}(u),
\end{align}
where, for each $k\geq 2$, $\mathcal{R}_{k}(u)$ is a polynomial vector field satisfying the bound
\begin{align}\label{eqn:bound_on_Rk_Cauchy_Kovalevskaya}
\Norm{\mathcal{R}_{k}(u)}{\deltax} \leq \tilde{C}_0 d^{k}C^k\Norm{u}{\deltax}^{k(d-2)+1}.
\end{align}
\end{lemma}
\begin{proof}[Proof of Lemma~\ref{lem:flow_of_polynomial_hamiltonian}] The existence of a locally analytic solution is a direct consequence of the standard Cauchy-Kovalevskaya theorem. The form of $\mathcal{R}_k$ and the bound \eqref{eqn:bound_on_Rk_Cauchy_Kovalevskaya} can be obtained by plugging \eqref{eqn:expression_of_solution_Cauchy_Kovalevskaya} into \eqref{eqn:differential_equation_Cauchy_Kovalevskaya} (we can differentiate the series by absolute convergence of the sum) and comparing terms order by order analogously to the final part of the proof of Proposition~\ref{prop:splitting_formulation_on_flows}.
\end{proof}
%
%
\begin{proof}[Proof of Proposition~\ref{prop:modified_polynomial_hamiltonian_expansion}] This proof is based on the proof of Theorem~IX.3.1 in \cite{hairer2013geometric}, but adapted to ensure each estimate is dimension-independent. 
To begin with, let us denote by $\Phi_{h,\varepsilon}$ the midpoint rule for the Hamiltonian system (with rescaled nonlinearity)
\begin{align*}
\frac{\dd u}{\dd t}=iAu+i\frac{\varepsilon}{h}f(u).
\end{align*}
Clearly, this implies that $\Phi_{h,\varepsilon}$ is a symplectic map for any choice of $h,\varepsilon$ for which $\Phi_{h,\varepsilon}$ is well-defined. By construction we have with the notation of the previous proof, see \eqref{eq:defF}
	\begin{align}\label{eqn:symplecticity_Psiheps_through_composition}
	\Phi_{h,\varepsilon} = R(hA) \circ \Psi^{h}_{\varepsilon} \quad \Longleftrightarrow \quad 
			\Psi^{h}_{\varepsilon}=R(hA)^{*}\circ\Phi_{h,\varepsilon},
	\end{align}
	{\em i.e.} $\Psi^{h}_{\varepsilon}$ is the composition of two symplectic maps ($R(hA)^{*}$ is the adjoint of the midpoint rule applied to the linear Schr\"odinger equation), hence symplectic.

 Moreover, the function $\Psi_{h,\varepsilon}$ admits, for a fixed $h >0$, a convergent expansion in powers of $\varepsilon$, see \eqref{eqn:expansion_psi_deltat_varepsilon} and the bound \eqref{boundpsikh} shows that the radius of convergence is independent of $h$. As symplectic $\varepsilon$-perturbation of the identity, we can thus construct a modified vector field at any order in $\varepsilon$ by following the classical method of \cite[Chapter IX]{hairer2013geometric} and shows that this modified vector is Hamiltonian. We then conclude by taking $\varepsilon = h$ small enough. Let us recall this construction, and show how the estimates obtained are independent of the parameters $K$ and $\deltax$. 
 
 We proceed by induction on $N$. To begin with, we note for $P_{0,h}$ given by \eqref{P0h}, we have $\Psi_{1,h} = i (\deltax)^{-1} \nabla_{\overline{u}}P_{0,h}$  where $\Psi_{h,1}$ is as defined in Proposition~\ref{prop:splitting_formulation_on_flows} and is given explicitly by  \eqref{firsterms}. Thus, by Lemma~\ref{lem:flow_of_polynomial_hamiltonian}, for any given $R>0$ there is a $\tau_0^{(0)}>0$ such that the corresponding Hamiltonian flow takes the form
   \begin{align*}
      \Phi^{\varepsilon}_{P_{0,h}}(u)=u+\varepsilon\Psi_{h,1}(u)+\sum_{\ell=2}^\infty\varepsilon^\ell\mathcal{R}^{(0)}_{h,\ell}(u)
 \end{align*}
 for all $0<\varepsilon,h\leq \tau_0^{(0)},u\in B_{\delta x}(R)$, where each $\mathcal{R}^{(0)}_{h,\ell}$ is a homogeneous polynomial vector field satisfying the bound
 \begin{align*}
    \Norm{ \mathcal{R}^{(0)}_{h,\ell}(u)}{\deltax}\leq \tilde{C}^{(0)}\left(C^{(0)}\right)^{\ell}\Norm{u}{\delta x}^{2r\ell+1},
 \end{align*}
for some $\tilde{C}^{(0)},C^{(0)}>0$ independent of $K,\deltax$. Taking $h=\varepsilon<\tau_0^{(0)}$ and using Proposition~\ref{prop:splitting_formulation_on_flows}, this shows Proposition~\ref{prop:modified_polynomial_hamiltonian_expansion} for the case $N=0$. For the induction step we now assume we have shown for a given $N$ that there are real-valued polynomials $P_{0,h},\dots,P_{N,h}$ such that for any $R>0$ there is $\tau_{0}^{(N)}>0$ such that for all $u\in B_{\deltax}(R)$ and $\varepsilon,h <\tau_0^{(N)}$		
 \begin{align}\label{eqn:induction_step_2}
     \Phi^{\varepsilon}_{P_{h,\varepsilon}^{(N)}}(u)=\Psi^\deltat_{\varepsilon}(u)+\sum_{\ell=N+2}^\infty \varepsilon^\ell\mathcal{R}^{(N)}_{h,\ell}(u),
 \end{align}
where $\Psi^\deltat_{\varepsilon}$ is as defined in \eqref{eqn:expansion_psi_deltat_varepsilon} and each $\mathcal{R}^{(N)}_{h,\ell}$ is a homogeneous polynomial vector field satisfying the bound
 \begin{align}\label{eqn:estimate_showing_analyticity_of_Phi_PN_2}
    \| \mathcal{R}^{(N)}_{h,\ell}(u)\|_{\deltax}\leq \tilde{C}^{(N)}\left(C^{(N)}\right)^{\ell}\|u\|_{\delta x}^{2r\ell+1},
 \end{align}
for some $\tilde{C}^{(N)},C^{(N)}>0$ independent of $K,\deltax$, and where
\begin{align}\label{eqn:def_P_h,eps}
    P_{h,\varepsilon}^{(N)}=P_{0,h}+\sum_{k=1}^{N}\varepsilon^k P_{k,h}.
\end{align}
We now seek to construct $P_{h,N+1}$ such that the analogous statement is true for $P_{h,\varepsilon}^{(N+1)}$. We note that using the estimate \eqref{eqn:estimate_showing_analyticity_of_Phi_PN} the series on the right hand side of \eqref{eqn:induction_step} is locally uniformly convergent, thus we can exchange order of differentiation and summation to find for $u\in B_{\deltax}(R)$ and $\varepsilon<\tau_{0}^{(N)}$ (using the notation introduced in \eqref{eqn:notation_Jacobian}):
\begin{align*}
    \nabla_{\overline{u}}^T\Phi^{\varepsilon}_{P_{h,\varepsilon}^{(N)}}(u)=\nabla_{\overline{u}}^T\Psi^{h}_{\varepsilon}(u)+\sum_{\ell=N+2}^\infty \varepsilon^\ell\nabla_{\overline{u}}^T\mathcal{R}^{(N)}_{h,\ell}(u).
\end{align*}
According to Lemma~\ref{lem:flow_of_polynomial_hamiltonian} and \eqref{eqn:symplecticity_Psiheps_through_composition} both $\Phi_{P_\varepsilon^{(N)}}^\varepsilon$ and $\Psi^h_{\varepsilon}$ are symplectic in the sense of Definition~\ref{def:symplectic_map}. Thus in slight abuse of notation we have (cf. \eqref{eqn:def_symplecticity}):
\begin{align*}
    i I=\left(\nabla_{\overline{u}}^T\Phi^{\varepsilon}_{P_{h,\varepsilon}^{(N)}}(u)\right)^Ti\nabla_{\overline{u}}^T\Phi^{\varepsilon}_{P_{h,\varepsilon}^{(N)}}(u)=\left(\nabla_{\overline{u}}^T\Psi^{h}_{\varepsilon}(u)\right)^Ti\nabla_{\overline{u}}^T\Psi^{h}_{\varepsilon}(u),
\end{align*}
where $I$ denotes the $(2K+1)\times (2K+1)$ identity matrix and $i$ represents multiplication by $-J$. Therefore, using the expression \eqref{eqn:induction_step}, we find
\begin{align}\label{eqn:leading_symplectic_correction}
    \varepsilon^{N+2}\left(\nabla_{\overline{u}}^T\Psi^{h}_{\varepsilon}(u)\right)^Ti\nabla_{\overline{u}}^T\mathcal{R}^{(N)}_{h,N+2}(u)+\varepsilon^{N+2}\left(\nabla_{\overline{u}}^T\mathcal{R}^{(N)}_{h,N+2}(u)\right)^Ti\nabla_{\overline{u}}^T\Psi^{h}_{\varepsilon}(u)+\varepsilon^{N+3}\mathcal{R}_{h,\varepsilon,N}(u)=0,
\end{align}
for some remainder $\mathcal{R}_{h,\varepsilon,N}$ uniformly bounded on $u\in B_{\deltax}(R), h,\varepsilon\in [0,\tau^{(N)}_0)$. We now note that $\Psi^{h}_{\varepsilon}(u)=u+\mathcal{O}(\varepsilon)$
\begin{align*}
    \nabla_{\overline{u}}^T\Psi^{h}_{\varepsilon}(u)=I+\mathcal{O}(\varepsilon),
\end{align*}
where the $\mathcal{O}(\varepsilon)$ represents terms which are uniformly bounded above by $C\varepsilon$ on $u\in B_{\deltax}(R)$ and with  $h,\varepsilon\in [0,\tau^{(N)}_0)$. Thus we have, by dividing \eqref{eqn:leading_symplectic_correction} by $\varepsilon^{N+2}$ and taking $\varepsilon\rightarrow 0$, that
\begin{align*}
    i\nabla_{\overline{u}}^T\mathcal{R}^{(N)}_{h,N+2}(u)=\left(i\nabla_{\overline{u}}^T\mathcal{R}^{(N)}_{h,N+2}(u)\right)^T,
\end{align*}
i.e. that the $(2K+1)\times (2K+1)$ matrix represented by $i\nabla_{\overline{u}}^T\mathcal{R}^{(N)}_{h,N+2}(u)$ is symmetric. Thus $\mathcal{R}^{(N)}_{h,N+2}(u)$ satisfies the assumptions of Lemma~\ref{lem:integration_symmetric_jacobian} and there is a real-valued homogeneous polynomial Hamiltonian $P_{N+1,h}$ such that 
\begin{align}\label{eqn:definition_of_P_{N+1,h}}
    - \mathcal{R}^{(N)}_{h,N+2}(u)=i(\deltax)^{-1}\nabla_{\bar{u}}P_{N+1,h},
\end{align}
and
\begin{align*}
|P_{N+1,h}(u)|\leq C\|u\|_{\deltax}^{2r(N+2)+2},\quad \forall u\in B_R.
\end{align*}
Let us now consider
\begin{align*}
P^{(N+1)}_\varepsilon:=P_{0,h}+\sum_{k=1}^{N+1}\varepsilon^k P_{k,h}.
\end{align*}
Then, it can be shown analogously to Lemma~\ref{lem:flow_of_polynomial_hamiltonian}, that there is $\tau_{0}^{(N+1)}\leq \tau_0^{(N)}$ such that for all $0\leq t<\tau_0^{(N+1)}, u\in B_{\deltax}(R)$ the flows $\Phi_{P_{\varepsilon}^{(N+1)}}^{t}(u),\Phi_{P_{\varepsilon}^{(N)}}^{t}(u)$ exist and in particular that $\Phi_{P_{\varepsilon}^{(N+1)}}^{t}(u)$ takes the following form
\begin{align}\label{eqn:expansion_flow_P^N+1}
    \Phi_{P_{\varepsilon}^{(N+1)}}^{t}(u)=u+\sum_{\ell,p=1}^\infty t^\ell \varepsilon^p  \mathcal{Q}^{(N+1)}_{h,\ell,p}(u),
\end{align}
where $\mathcal{Q}^{(N+1)}_{h,\ell,p}$ are homogeneous polynomial vector fields satisfying
\begin{align*}
    \Norm{\mathcal{Q}^{(N+1)}_{h,\ell,p}(u)}{\deltax}\leq C_N^{\ell+p}\Norm{u}{\deltax}^{2r(\ell+p)+1}
\end{align*}
for some $C_N>0$ depending only on $N$. Note that by taking $t = \varepsilon$, this implies the existence of the expansion \eqref{eqn:induction_step} at order $N+1$, and we only have to show bounds on the remainder term and the cancellation of the term of order $N+2$. 

For $\varepsilon \leq \tau_0^{(N+1)}$, we 
introduce two maps $g,f:B_{\deltax}(R)\rightarrow\mathbb{C}^{2K+1}$ by
\begin{align*}
    g_\varepsilon(u)&:=i(\deltax)^{-1}\nabla_{\bar{u}}P_{\varepsilon}^{(N)},\\
    f_\varepsilon(u)&:=i(\deltax)^{-1}\nabla_{\bar{u}}P_{\varepsilon}^{(N+1)}.
\end{align*}
For this choice we have
\begin{align}\label{eqn:DE_satisfied_by_g}
    \frac{\dd \Phi^{t}_{P_{\varepsilon}^{(N)}}}{\dd t}&=g_\varepsilon(\varphi^{t}_{P_{\varepsilon}^{(N)}}),\quad 0<t,h,\varepsilon<\tau_0^{(N+1)},\\\label{eqn:DE_satisfied_by_f}
    \frac{\dd \Phi^{t}_{P_{\varepsilon}^{(N+1)}}}{\dd t}&=f_\varepsilon(\varphi^{t}_{P_{\varepsilon}^{(N+1)}}),\quad 0<t,h,\varepsilon<\tau_0^{(N+1)}.
\end{align}
We proceed in two steps. Firstly, we shall show that there is $C>0$ such that for all $u\in B_{\deltax}(R),0\leq t,\varepsilon,h<\tau_0^{(N+1)}$:
\begin{align}\label{eqn:order_of_difference_of_flows}
    \|\Phi^{t}_{P_{\varepsilon}^{(N)}}(u)-\Phi^{t}_{P_{\varepsilon}^{(N+1)}}(u)\|_{\deltax}\leq t\varepsilon^{N+1} C \Norm{u}{\delta x}^{2r(N+2) + 1}.
\end{align}
We note that
\begin{align*}
    \Norm{f_\varepsilon(\Phi^{t}_{P_{\varepsilon}^{(N+1)}}(u))-g_\varepsilon(\Phi^{t}_{P_{\varepsilon}^{(N+1)}}(u))}{\deltax}=\varepsilon^{N+1}\Norm{\mathcal{R}_{h,N+2}^{(N)}(\Phi^{t}_{P_{\varepsilon}^{(N+1)}}(u))}{\deltax}\leq \varepsilon^{N+1} C \Norm{u}{\deltax}^{2r(N+2)+1},
\end{align*}
for some $C>0$ so long as $u\in B_{\deltax}(R)$. This holds true because $\mathcal{R}_{h,N+2}^{(N)}$ is a polynomial vector field in its argument with bounded coefficients and the local well-posedness of the Hamiltonian system corresponding to $P_{\varepsilon}^{(N+1)}$ established in \eqref{eqn:expansion_flow_P^N+1}. Moreover, $g$ is a polynomial with bounded coefficients, and it follows immediately that it is Lipschitz continuous (see Lemma~\ref{lem:stability_estimate_polynomial_vector_field}) with a constant independent of $K,\deltax$, {\em i.e.} there is a $C>0$ such that for $u,v\in B_{\deltax}(R), \varepsilon\in[0,\tau_{0}^{(N+1)})$ we have
\begin{align*}
\Norm{g_\varepsilon(u)-g_\varepsilon(v)}{\deltax}\leq C \big( \Norm{u}{\deltax} + \Norm{v}{\deltax} \big)^{2r}\Norm{u-v}{\deltax}
\end{align*}
uniformly in $h$, $\varepsilon$ small enough. 
Thus  we have (noting that the flows are equal at $t=0$)
\begin{align*}
   \Norm{\Phi^{t}_{P_{\varepsilon}^{(N+1)}}(u)-\Phi^{t}_{P_{\varepsilon}^{(N)}}(u)}{\deltax}\leq C_2\int_0^t\varepsilon^{N+1} \Norm{u}{\delta x}^{2r(N+2) + 1} \dd s\leq C_2 t \varepsilon^{N+1} \Norm{u}{\delta x}^{2r(N+2) + 1},
\end{align*}
for some $C_2>0$ which depends on $R$ but not on $K,\deltax$, thus completing the estimate \eqref{eqn:order_of_difference_of_flows}. 

\medskip
\noindent In the second step we will show that
\begin{align}\label{eqn:form_of_difference_of_flows}
    \|\Phi^{t}_{P_{t}^{(N+1)}}(u)-\Phi^{t}_{P_{t}^{(N)}}(u)+t \varepsilon^{N+1}\mathcal{R}_{h,N+2}^{(N)}(u)\|_{\deltax}\leq C t^{N+3}\Norm{u}{\deltax}^{2(N+3)+1},
\end{align}
for some constant $C>0$ independent of $K,\deltax$. For this we note that by \eqref{eqn:DE_satisfied_by_g} \& \eqref{eqn:DE_satisfied_by_f} we have
\begin{align}\nonumber
    \Phi^{t}_{P_{\varepsilon}^{(N+1)}}(u)-\Phi^{t}_{P_{\varepsilon}^{(N)}}(u)&=\int_0^{t}f_\varepsilon(\Phi^{s}_{P_{\varepsilon}^{(N+1)}}(u))-g_\varepsilon(\Phi^{s}_{P_{\varepsilon}^{(N)}}(u))\dd s\\
    \begin{split}\label{eqn:expression_updated_flows1}
    &=\int_0^{t}f_\varepsilon(\Phi^{s}_{P_{\varepsilon}^{(N+1)}}(u))-g_\varepsilon(\Phi^{s}_{P_{\varepsilon}^{(N+1)}}(u))\dd s\\
    &\quad+\int_0^tg_\varepsilon(\Phi^{s}_{P_{\varepsilon}^{(N+1)}}(u))-g_\varepsilon(\Phi^{s}_{P_{\varepsilon}^{(N)}}(u))\dd s.
    \end{split}
\end{align}
Now we note that
\begin{align}\nonumber
    \int_0^{t}f_\varepsilon(\Phi^{s}_{P_{\varepsilon}^{(N+1)}}(u))-g_\varepsilon(\Phi^{s}_{P_{\varepsilon}^{(N+1)}}(u))\dd s&= -\int_{0}^{t} \varepsilon^{N+1}\mathcal{R}_{h,N+2}^{(N)}(\Phi^{s}_{P_{\varepsilon}^{(N+1)}}(u)) \dd s\\\label{eqn:expression_updated_flows2}
    &=-\mathcal{R}_{h,N+2}^{(N)}(u) \int_{0}^{t} \varepsilon^{N+1}\dd s + t^2 \varepsilon^{N+2}\mathcal{Q}^{(1)}_{\deltat,\varepsilon,t}(u),
\end{align}
where $\mathcal{Q}^{(1)}_{\deltat,\varepsilon,t}$ is a function of $u$ bounded uniformly on $B_{\deltax}(R), (\deltat,\varepsilon,t)\in [0,\tau_0^{(N+1)}]$ and is bounded by $C_N \Norm{u}{\deltax}^{2(N+3)+1}$ with a convergent expansion of the form \eqref{eqn:expansion_flow_P^N+1}. 
In this final line we used the observation from \eqref{eqn:expansion_flow_P^N+1} that
\begin{align*}
    \varphi^{s}_{P_{\varepsilon}^{(N+1)}}(u)=u+t\mathcal{Q}_{\deltat,\varepsilon,t}^{(2)}(u),
\end{align*}
where $\mathcal{Q}_{\deltat,\varepsilon}^{(2)}(t,u)$ is some analytic function of $u$ bounded uniformly on $u\in B_{\deltax}(R), h,\varepsilon,t \in[0,\tau_0^{(N+1)}]$. Thus, combining \eqref{eqn:expression_updated_flows1} \& \eqref{eqn:expression_updated_flows2} we obtain that given $R>0$ there is a constant $C_2$ depending only on $N$ such that
\begin{align*}
  \left\|\Phi^{t}_{P_{\varepsilon}^{(N+1)}}(u)-\Phi^{t}_{P_{\varepsilon}^{(N)}}+t \varepsilon^{N+1}\mathcal{R}_{h,N+2}^{(N)}(u)\right\|_{\deltax}&\leq t^2 \varepsilon^{N+1}\Norm{u}{\deltax}^{2(N+3)+1}\\
  &+C t \Norm{u}{\delta x}^{2r}\left\|\Phi^{t}_{P_{\varepsilon}^{(N+1)}}(u)-\Phi^{t}_{P_{\varepsilon}^{(N)}}\right\|_{\deltax}\\
  &\leq t^2 \varepsilon^{N+1} C_2\Norm{u}{\deltax}^{2(N+3)+1}
\end{align*}
for any $u\in B_{\deltax}(R)$. This completes the proof of \eqref{eqn:form_of_difference_of_flows}, hence the induction step and therefore the proof of the desired result (by taking again $h=t=\varepsilon<\tau_{0}^{(N)}$ for any given choice of $N$).
\end{proof}
\section{Modified energy for the midpoint rule}\label{sec:modified_energy_for_midpoint_rule}
Proposition~\ref{prop:modified_polynomial_hamiltonian_expansion} essentially shows that the midpoint rule can (to arbitrary desired order) be written as a Hamiltonian splitting method. In this section we will establish our central result, Theorem~\ref{thm:modified_energy_midpoint} which shows the existence of a modified energy for the midpoint rule, exploiting this ``approximate'' splitting formulation. We begin by introducing the vector field associated with the linear part of the flow of \eqref{eqn:midpoint_rule_nlse}.

\begin{lemma}\label{lem:estimate_A0}
    The vector field $X_{A_0}$ associated with the flow $\Phi_{A_0}^1=R(hA)$ is given by
    \begin{align}\label{eqn:X_A_0}
        X_{A_0}=2iU^{-1}\arctan\left(\frac{hD}{2}\right)U,
    \end{align}
where $U,D$ are as in \eqref{eqn:def_RdeltatA}. Furthermore it satisfies the estimate
\begin{align}\label{eqn:estimate_norm_X_A_0}
        \Norm{X_{A_0}u}{\deltax}&\leq 2\arctan\left(\frac{h}{\deltax^2 }\right)\Norm{u}{\deltax}.
    \end{align}
\end{lemma}
\begin{proof}[Proof of Lemma~\ref{lem:estimate_A0}]
    We note by linearity we can integrate \eqref{eqn:X_A_0} over $t\in [0,1]$ to directly recover $R(hA)$. By \eqref{eqn:bilinear_expression_discrete_norm} the norm $\Norm{\,\cdot\,}{\deltax}$ and $A$ can be jointly diagonalised meaning that 
    \begin{align*}
        \Norm{X_{A_0}u}{\deltax}\leq \max_{1\leq j\leq 2K+1}\left|2\arctan\left(\frac{h\lambda_j}{2}\right)\right|\Norm{u}{\deltax},
    \end{align*}
    where $\lambda_j, j=1,\dots,2K+1$ are the eigenvalues of $A$. Since $A$ is a symmetric tridiagonal Toeplitz matrix, the eigenvalues are well-known to be
    \begin{align}
    \label{diagonal}
        \lambda_j=-\frac{1}{\deltax^2}\left(2-2\cos\left(\frac{j\pi}{2K+2}\right)\right),
    \end{align}
    whence the estimate \eqref{eqn:estimate_norm_X_A_0} immediately follows.
\end{proof}

\subsection{Formal construction of the modified energy} Using the above basics we can now formally construct the modified energy for the midpoint rule. We shall rigorously show that the flow corresponding to this modified energy corresponds (up to arbitrary desired order) to the one of the midpoint rule in Section~\ref{sec:analytic_estimates}. We take a similar approach to \cite{faou2011hamiltonian,bambusi2013existence}, but note that in our case we have to include a second small parameter $\varepsilon$ which captures the expansion of the modified vector field corresponding to the nonlinear part of the midpoint rule \eqref{eqn:def_P_h,eps}, to avoid the confusion with the $h$ appearing in the terms depending on $hA$. Thus, we look for a real Hamiltonian function $Z(t, \eps;u)$ such that
\begin{align}\label{eqn:equality_of_flows}
    \Phi_{Z(t,\eps)}^{1}=\Phi^{1}_{A_0}\circ\Phi^{t}_{P_{h,\varepsilon}^{(N)}}, \quad \forall t,\varepsilon<\tilde{\tau}^{(N)}_0,
\end{align}
for some threshold $\tilde{\tau}^{(N)}_0>0$ to be determined, but which should be independent of $K,\deltax$. In the above, $A_0$ is as defined in Lemma~\ref{lem:estimate_A0} and $P_{h,\eps}^{(N)}$ is as defined in \eqref{eqn:def_P_h,eps}. According to \cite[Section 3]{faou2011hamiltonian} we have, formally,
\begin{align}
\label{eqformal}
    \frac{\partial}{\partial t}\Phi_{Z(t,\eps)}^{1}=X_{Q(t,\eps)}\circ\Phi_{Z(t,\eps)}^{1},
\end{align}
where the modified vector field $X_Q$ has the formal series \cite[(3.3)]{faou2011hamiltonian}
\begin{align}\label{eqn:formal_series_Q}
  X_{Q(t,\eps)}=\sum_{k\geq 0}\frac{1}{(k+1)!}\ad_{X_{Z(t,\eps)}}^k\partial_t X_{Z(t,\eps)}.
\end{align}
To establish an expression for $Z$ we would also like to differentiate the right hand side of \eqref{eqn:equality_of_flows}. We have
\begin{align}\label{eqn:time_derivative_composition}
    \frac{\dd }{\dd t}\left(\Phi_{A_0}^1\circ\Phi_{P_{h}^{(N)}}^t(u)\right)=d\Phi_{A_0}^1\vert_{p=\Phi_{P_{h}^{(N)}}^t(u)}\circ X_{P_{h}^{(N)}}\circ\Phi_{P_{h}^{(N)}}^t(u),
\end{align}
where, by linearity, we actually have $\dd\Phi_{A_0}^1\vert_{p=\Phi_{P_{h}^{(N)}}^t(u)}=R(hA)$. Moreover, we recall that $X_{P}=i\deltax^{-1}\nabla_{\bar{u}} P$, thus we observe that for $\tilde{P}=P\circ R(hA)^*$, i.e. $$\tilde{P}(u,\bar{u})=P(R(hA)^*u,R(hA)^T\bar{u}),$$ we find in coordinates
\begin{align}\nonumber
    (i\deltax^{-1}\nabla_{\bar{u}} \tilde{P})_{j}&=i\deltax^{-1}\frac{\partial \tilde{P}}{\partial \bar{u}_j}=i\deltax^{-1} \sum_{k,\ell} \frac{\partial P}{\partial \bar{u}_k}\frac{\partial R(hA)^T_{kl}\bar{u}_{l}}{\partial \bar{u}_{j}}\\\label{eqn:linear_action_on_vector_field}
    &=\sum_k R(hA)_{jk}i\deltax^{-1}\frac{\partial P}{\partial \bar{u}_k}=(R(hA)\circ X_{P})_j.
\end{align}
And thus, combining this with \eqref{eqn:time_derivative_composition} gives
\begin{align}\label{eqn:time_derivative_vector_fields2}
\frac{\dd }{\dd t}\left(\Phi_{A_0}^1\circ\Phi_{P_{h,\eps}^{(N)}}^t(u)\right)\Big\vert_{t=t}&=X_{P_{h,\eps}^{(N)}\circ R(hA)^*}\circ\Phi_{A_0}^1\circ\Phi_{P_{h,\eps}^{(N)}}^t(u)\\
&= X_{P_{h,\eps}^{(N)}\circ R(hA)^*}\circ  \Phi_{Z(t,\eps)}^{1}. \nonumber
\end{align}
Thus combining \eqref{eqn:formal_series_Q} \& \eqref{eqn:time_derivative_vector_fields2} yields the following equation to be satisfied by $X_{Z(t,\eps)}$:
\begin{align}\label{eqn:implicit_expressionX_Z}
 \sum_{k\geq0}\frac{1}{(k+1)!}\ad_{X_{Z(t,\eps)}}^k\partial_tX_{Z(t,\eps)}=X_{P_{h,\eps}^{(N)}\circ R(hA)^*},
\end{align}
where $\ad_{X_H}(X_{G}):=[X_H,X_G]$. We shall now seek a formal solution to this equation. We start by noting that \eqref{eqn:implicit_expressionX_Z} is formally equivalent to 
\begin{align*}
    \partial_{t}X_{Z(t,\eps)}=\sum_{k\geq 0}\frac{B_k}{k!}\ad_{X_{Z(t,\eps)}}^kX_{P_{h,\eps}^{(N)}\circ R(hA)^*},
\end{align*}
where $B_k$ are the Bernoulli numbers. Note this inversion is only valid for the full series when $\|\ad_{X_{Z(t,\eps)}}^k\partial_tX_{Z(t,\eps)}\|_{\deltax}<(2\pi)^k$, but the term-by-term conditions are algebraically equivalent regardless of the size of $\|\ad_{X_{Z(t,\eps)}}^k\partial_tX_{Z(t,\eps)}\|_{\deltax}$. In order to find (a suitable truncation) of $X_Z$ we try the following Ansatz
\begin{align}\label{eqn:formal_ansatz_Z}
    Z_{t,\eps}=\sum_{\ell,j=0}^{\infty}t^{\ell}\eps^j Z_{\ell,j,h},
\end{align}
with 
\begin{align*}
    Z_{0,j,h}:=\begin{cases}
        A_0,&\quad \text{if $j=0$},\\
        0,&\quad\text{otherwise}.
    \end{cases}
\end{align*}
This translates to an equivalent Ansatz for $X_{Z(t,\eps)}$ by the results in Section~\ref{sec:vector_fields_and_Hamiltonian_fns}. For notational simplicity we will in the following write $Z_{\ell,j}$ for $Z_{\ell,j,h}$. Matching terms formally order-by-order in both $t$ and $\eps$ this leads to the recursion
\begin{align}\label{eqn:recursion_Zlj}
    (\ell+1)X_{Z_{\ell+1,j}}=\sum_{k\geq 0}\frac{B_k}{k!}\sum_{m=0}^{\min\{j,N\}}\sum_{\substack{\ell_1+\ell_2+\cdots+\ell_k=\ell\\j_1+j_2+\cdots+j_k=j-m}}\ad_{X_{Z_{\ell_1,j_1}}}\cdots \ad_{X_{Z_{\ell_k,j_k}}} X_{P_{m,h}\circ R(hA)^*},
\end{align}
where $P_{m,h}$ are as constructed in Proposition~\ref{prop:modified_polynomial_hamiltonian_expansion}.
\begin{lemma}\label{lem:estimate_on_size_Zlj}
    Fix $N\in\mathbb{N}$ and suppose we have for some $\tilde{\epsilon}>0$ that the following CFL condition is satisfied
    \begin{align}\label{eqn:CFL_condition_modified_energy}
   h\leq \deltax^2 \tan\left(\frac{\pi-\tilde{\epsilon}}{2(2r(N+1)+1)}\right).
\end{align}
    Then, for every $0\leq \ell,j$ with $\ell+j\leq N$, the Hamiltonian $Z_{\ell,j}$ and associated vector field $X_{Z_{\ell,j}}$ is uniquely defined by \eqref{eqn:recursion_Zlj}, with the infinite series on the right hand side of \eqref{eqn:recursion_Zlj} existing and converging. Moreover, $X_{Z_{\ell,j}}\in\mathcal{P}_{2r(\ell+j)+1}$,  and $X_{Z_{\ell,j}}=0$ whenever $j> N\ell$ and there is a constant $C_{\ell,j}>0$ depending only on $N,\ell,j$ such that for all $K,\deltax>0$,
    \begin{align}\label{eqn:estimate_X_Zlj}
        \Norm{X_{Z_{\ell,j}}}{\deltax}\leq C_{\ell,j}.
    \end{align}
\end{lemma}
\begin{proof}[Proof of Lemma~\ref{lem:estimate_on_size_Zlj}] We shall prove the result by induction on $\ell$. Clearly $X_{Z_{0,j}}\in\mathcal{P}_{2rj+2}$ for all $j\geq 0$ and $X_{Z_{0,j}}=0$ when $j>1$. For the induction step we proceed as follows. Firstly, we note that the only non-zero contributions to $Z_{\ell+1,j}$ in \eqref{eqn:recursion_Zlj} arise if $j_i\leq N\ell_i, i=1,\dots,k$, i.e. if
\begin{align*}
    j=j_1+\cdots+j_k+m\leq \ell_1 N+\cdots \ell_kN+m\leq \ell N+N,
\end{align*}
i.e. if $j\geq N(\ell+1)$, so $X_{Z_{\ell+1,j}}=0,$ whenever $j>N(\ell+1)$. Next we note that any non-zero contribution to $X_{Z_{\ell,j}}$ is a polynomial of degree $\leq 2r(\ell+j+1)+1$, since
\begin{align*}
  \deg \left( \ad_{X_{Z_{\ell_1,j_1}}}\cdots \ad_{X_{Z_{\ell_k,j_k}}} X_{P_{m,h}}\right)=2r(m+1)+1-2k+\sum_{q=1}^{k}\deg Z_{\ell_q,j_q}
\end{align*}
and so, if $\ell_1+\ell_2+\cdots+\ell_k=l$, $j_1+j_2+\cdots+j_k=j-m$, then we have
\begin{align*}
     \deg \left( \ad_{X_{Z_{\ell_1,j_1}}}\cdots \ad_{X_{Z_{\ell_k,j_k}}} X_{P_{m,h}}\right)&\leq 2r(m+1)+1-2k+2r(\ell+j-m)+2k\\&=2r(\ell+j+1)+1.
\end{align*}
Thus, so long as the term on the right hand side of \eqref{eqn:recursion_Zlj} converges in $\Norm{\cdot}{\deltax}$, $X_{Z_{\ell+1,j}}\in\mathcal{P}_{2r(\ell+1+j)+1}$. To show this, we can use the estimate \eqref{eqn:estimate_poisson_bracket}:
\begin{align*}
    (\ell+1)&\Norm{X_{Z_{\ell+1,j}}}{\deltax}\\
    &\leq \sum_{k\geq 0} \frac{B_k}{k!}\sum_{m=0}^{\min\{j,N\}}\sum_{\substack{\ell_1+\ell_2+\cdots+\ell_k=l\\j_1+j_2+\cdots+j_k=j-m}}\left((2r(\ell_1+j+1)+2r(\sum_{i=2}^k \ell_i+j_i+m+1)+3)\right)\cdots\\
    &\quad\cdots\left((2r(\ell_{k-1}+j_{k-1})+2r(\ell_k+j_k+m+1)+3)\right)(2r(\ell_k+j_k)+2r(m+1)+3)\\
    &\quad\quad\Norm{X_{Z_{\ell_1,j_1}}}{\deltax}\cdots \Norm{X_{Z_{\ell_k,j_k}}}{\deltax}\Norm{X_{P_{m,h}\circ R(hA)^*}}{\deltax}.
\end{align*}
Let us denote by $d_{\max}$ the largest degree appearing in the above estimate, i.e.
\begin{align*}
    d_{\max}=2r(\ell+j-m+m+1)+1=2r(\ell+j+1)+1.
\end{align*}
Then we can simplify the estimate to
\begin{eqnarray*}
 &&   (\ell+1)\Norm{X_{Z_{\ell+1,j}}}{\deltax}\\
    &\leq & \sum_{k\geq 0} \frac{B_k}{k!}\sum_{m=0}^{\min\{j,N\}}(2d_{\max})^{k}\sum_{\substack{\ell_1+\ell_2+\cdots+\ell_k=l\\j_1+j_2+\cdots+j_k=j-m}}\Norm{X_{Z_{\ell_1,j_1}}}{\deltax}\cdots \Norm{X_{Z_{\ell_k,j_k}}}{\deltax}\Norm{X_{P_{m,h}\circ R(hA)^*}}{\deltax}\\
    &\leq& \sum_{k\geq 0}\frac{B_k}{k!}\sum_{m=0}^{\min\{j,N\}}(2d_{\max})^k\sum_{\tilde{r}=1}^{\ell}\binom{k}{k-\tilde{r}}\Norm{X_{A_0}}{\deltax}^{k-\tilde{r}} \\
    &&\times \sum_{\substack{\ell_i>0,\ell_1+\ell_2+\cdots+\ell_{\tilde{r}}=\ell\\j_1+j_2+\cdots+j_{\tilde{r}}=j-m}}\Norm{X_{Z_{\ell_1,j_1}}}{\deltax}\cdots \Norm{X_{Z_{\ell_{\tilde{r}},j_{\tilde{r}}}}}{\deltax}\Norm{X_{P_{m,h}\circ R(hA)^*}}{\deltax}
    \end{eqnarray*}
    and thus
    \begin{eqnarray*}
   (\ell+1)\Norm{X_{Z_{\ell+1,j}}}{\deltax} &\leq& \sum_{m=0}^{\min\{j,N\}}\sum_{\tilde{r}=1}^{\ell}\left(\sum_{k\geq 0}\frac{B_k}{k!}(2d_{\max})^k\binom{k}{k-\tilde{r}}\|X_{A_0}\|_{\deltax}^{k-\tilde{r}}\right)\\
    &&\times\sum_{\substack{\ell_i>0,\ell_1+\ell_2+\cdots+\ell_{\tilde{r}}=\ell\\j_1+j_2+\cdots+j_{\tilde{r}}=j-m}}\Norm{X_{Z_{\ell_1,j_1}}}{\deltax}\cdots \Norm{X_{Z_{\ell_{\tilde{r}},j_{\tilde{r}}}}}{\deltax}\Norm{X_{P_{m,h}\circ R(hA)^*}}{\deltax}.
    \end{eqnarray*}
Now we know that
\begin{align*}
    f:x\mapsto \sum_{k\geq 0} \frac{B_k}{k!}(2d_{\max})^kx^k
\end{align*}
has radius of convergence $|x|<2\pi/(2d_{\max})$ and that
\begin{align*}
    \frac{\dd^{\tilde{r}}}{\dd x^{\tilde{r}}}f(x)=\sum_{k\geq 0} \frac{B_k}{k!}(2d_{\max})^k\frac{k!}{(k-{\tilde{r}})!}x^{k-\tilde{r}}.
\end{align*}
Thus we have
\begin{align}\nonumber
    (\ell+1)\Norm{X_{Z_{\ell+1,j}}}{\deltax}&\leq  \tilde{C}(\ell+1)\sum_{m=0}^{\min\{j,N\}}\Norm{X_{P_{m,h}}}{\deltax}\\
    &\nonumber\hskip 3ex \times\sum_{\tilde{r}=1}^{\ell}\sum_{\substack{\ell_i>0,\ell_1+\ell_2+\cdots+\ell_{\tilde{r}}=\ell\\j_1+j_2+\cdots+j_{\tilde{r}}=j-m}}\underbrace{\Norm{X_{Z_{\ell_1,j_1}}}{\deltax}\cdots \Norm{X_{Z_{\ell_{\tilde{r}},j_{\tilde{r}}}}}{\deltax}}_{\leq \max_{\ell_k\leq \ell,j_k\leq j}C_{l_k,j_k}^{l+j-m}}\\\label{eqn:intermediate_estimate1}
    &\leq (\ell+1)\tilde{C}\sum_{m=0}^{\min\{j,N\}}\|X_{P_{m,h}\circ R(hA)^*}\|_{\deltax}\max_{\ell_k\leq \ell,j_k\leq j}C_{\ell_k,j_k}^{\ell+j-m}\sum_{\tilde{r}=1}^{\ell}\sum_{\substack{\ell_i>0,\ell_1+\ell_2+\cdots+\ell_{\tilde{r}}=\ell\\j_1+j_2+\cdots+j_{\tilde{r}}=j-m}}1
\end{align}
where $\tilde{C}=\sup_{|x|<2\pi-2\tilde{\epsilon}, {\tilde{r}}=1,\dots, l}|f^{({\tilde{r}})}(x)|$, provided that
\begin{align}\label{eqn:required_bound_on_X_A_0}
    \Norm{X_{A_0}}{\deltax}\leq \frac{\pi-\tilde{\epsilon}}{d_{\max}}.
\end{align}
By Lemma~\ref{lem:estimate_A0}, and since $d_{\max}\leq 2r(\ell+j+1)+1$, the CFL condition \eqref{eqn:CFL_condition_modified_energy} is sufficient to guarantee \eqref{eqn:required_bound_on_X_A_0}. Thus we have by \eqref{eqn:intermediate_estimate1} that
\begin{align*}
\Norm{X_{Z_{\ell+1,j}}}{\deltax}\leq C_{\ell+1,j}\max_{0\leq m\leq \min\{j,N\}}\Norm{X_{P_{m,h}}}{\deltax},
\end{align*}
for some constant $C_{\ell+1,j}>0$ independent of $K,\deltax$. Moreover, by construction  (cf. \eqref{eqn:definition_of_P_{N+1,h}}) we have 
\begin{align*}
    \Norm{X_{P_{m,h}\circ R(hA)^*}(u)}{\deltax}\leq C_{m}\Norm{u}{\deltax}^{2r(m+1)+1}, \quad \forall u\in V_{\deltax},
\end{align*}
for some constants $C_m>0$ independent of $\deltax, K$. The result thus immediately follows from Lemma~\ref{lem:estimates_multilinear_norm}.
\end{proof}

\subsection{Final estimates and modified energy}\label{sec:analytic_estimates}
We are now in a position to show that the modified energy defined in the above is indeed an accurate representation of the midpoint rule, i.e. we are able to prove the following central result of the present work.
\begin{theorem}
\label{thm:modified_energy_midpoint}

Let $u^{n+1} = \varphi_h(u^n)$ be the application defined by the midpoint rule \eqref{eqn:midpoint_rule_nlse}. Fix $N \in \N$. Then there exists $h_0$ such that for 
$h \leq h_0$ and all $K$ and $\deltax$ satisfying the condition CFL condition \eqref{eqn:CFL_condition_modified_energy} then there exists 
\begin{equation}
\label{energy}
H_h = A_0 + B_h,
\end{equation}
with 
$$
A_0(u) = \bar u^T  \frac{2}{h} \arctan \left( \frac{h A}{2}\right) u,
$$
and
\begin{equation}
\label{decoB}
B_h(u) =  \sum_{k = 0}^N  h^k B_{k,h},
\end{equation}
where the $B_{k,h}$ are polynomials of degree $2r(k+1) + 2$, and
 such that there is a constant $C_N>0$ independent of $K,\deltax>0$ such that  $\forall R \leq 1$, for all $u \in B_{\delta x}(R)$, then 
\begin{align}
\label{mainest}
\left\|\varphi_h(u) - \Phi^h_{H_h}(u)\right\|_{\deltax}\leq h^{N+2} C_N R^{2r (N+2)  +  1 }, 
\end{align}
 i.e. $H_h$ is the modified energy up to $\mathcal{O}(h^{N+2})$. 

\end{theorem}
\begin{proof}[Proof of Theorem~\ref{thm:modified_energy_midpoint}] The proof of this theorem is largely the same as Theorem 4.2 in \cite{faou2011hamiltonian}, but with the double expansion in $t$ and $\varepsilon$. Let us stress the main points: 

\noindent {\em (i) \underline{Reduction to the splitted form}}. 
The combination of Proposition \ref{prop:splitting_formulation_on_flows} and Proposition \eqref{prop:modified_polynomial_hamiltonian_expansion}
 shows that the Theorem is equivalent to proving the same result with the flow 
 $$
 \Phi^{1}_{A_0}\circ\Phi^{t}_{P^{(N)}_{h,\varepsilon}}
 $$ 
 instead of the midpoint rule $\varphi_h$. 
 
 \noindent {\em (ii) \underline{Construction of the modified Hamiltonian}}. 
We define the truncated expansion $Z^{(N)}(t,\varepsilon)$ of \eqref{eqn:formal_ansatz_Z} by
\begin{align*}
    Z^{(N)}(t,\varepsilon):=\sum_{\ell+j=0}^{N}t^l\varepsilon^jZ_{\ell,j,h}.
\end{align*}
and we set $H_h = Z^{(N)}(h,h)$. Note that we have indeed $A_0 = Z_{0,0,h}$ and that the 
$$
B_{k,h} := \sum_{\ell + j = k} Z_{k,j,h}
$$
are homogeneous polynomials of order $2r(N+1) + 2$. 

\noindent {\em (iii) \underline{Flow of the modified equation }}. 
We define the Hamiltonian $Q^{(N)}(t,\varepsilon)$ by the formula
\begin{align*}
 \sum_{k\geq0}\frac{1}{(k+1)!}\ad_{X_{Z^{(N)}(t,\eps)}}^k\partial_tX_{Z^{(N)}(t,\eps)}=X_{Q^{(N)}(t,\varepsilon) } . 
\end{align*}
In view of \eqref{eqn:implicit_expressionX_Z} and the construction of the $Z$, we can prove that  
\begin{align}
\label{raslbol}
&Q^{(N)}(t,\varepsilon)(u)  =  P_{h,\eps}^{(N)} \circ R(hA)^*(u)  + O( h^{N+1}R^{2r(N+2) + 2}),
\end{align}
for $|t| \leq h, |\varepsilon| \leq h,$ and $u \in B_{\deltax}(R)$. Similarly to  Lemma 4.3 in \cite{faou2011hamiltonian}, we can prove that for $t,\varepsilon \leq \tau_0(N)$ and for $R \leq 1$, small enough, we have (cf. the formal equation \eqref{eqformal}) 
   \begin{align*}
        \frac{\dd}{\dd t}\Phi_{Z^{(N)}(t,\varepsilon)}(u)=X_{Q^{(N)}(t,\varepsilon)}\circ \Phi_{Z^{(N)}(t,\varepsilon)}(u)
    \end{align*}
for $u \in B_\deltax(R)$. 

 \noindent {\em (iv) \underline{Estimating the difference}}. We set 
 \begin{align*}
v(t):=\Phi^{1}_{A_0}\circ\Phi^{t}_{P^{(N)}_{h,\varepsilon}}-\Phi_{Z^{(N)}(t,\varepsilon)}^{1},
\end{align*}
and we want to control this quantity uniformly in $|t| \leq h$ and $|\varepsilon|\leq h$. 
To do this, we write 
\begin{align*}
    \|v(t)\|_{\deltax}&\leq \int_{0}^t\left\| X_{P_{h,\eps}^{(N)}\circ R(hA)^*}\circ\Phi_{A_0}^1\circ\Phi_{P_{h,\eps}^{(N)}}^t(v(s))-X_{Q(t,\eps)}\circ\Phi_{Z^{(N)}(t,\eps)}^{1}(v(s))\right\|_{\deltax}\dd s\\
    &\leq \underbrace{\int_{0}^t\left\| X_{P_{h,\eps}^{(N)}\circ R(hA)^*}\circ\Phi_{Z^{(N)}(t,\eps)}^{1}(v(s))-X_{Q(t,\eps)}\circ\Phi_{Z^{(N)}(t,\eps)}^{1}(v(s))\right\|_{\deltax}\dd s}_{=:A}\\
    &\quad + \underbrace{\int_{0}^t\left\| X_{P_{h,\eps}^{(N)}\circ R(hA)^*}\circ\Phi_{A_0}^1\circ\Phi_{P_{h,\eps}^{(N)}}^t(v(s))-X_{P_{h,\eps}^{(N)}\circ R(hA)^*}\circ\Phi_{Z^{(N)}(t,\eps)}^{1}(v(s))\right\|_{\deltax}\dd s}_{=:B}
\end{align*}
Then 
\begin{itemize}
\item $A$ is estimated using \eqref{raslbol} and we have that 
$$
A \leq C_N h^{N+2} R^{2r(N+2) + 2}
$$
for $u \in  B_\deltax(R)$ and $R \leq 1$, as in this case $v(s) \in B_\deltax(2R)$ for $s \in [0,h]$. 

\item $B$ is estimated as the Hamiltonian vector field $X_{P_{h,\eps}^{(N)}\circ R(hA)^*}$ is locally Lipschitz (with constant depending on $N$, but uniformly in $|t| \leq h$ and $|\varepsilon|\leq h$. 
\end{itemize}
The conclusion then follows from classical Gronwall estimates \cite{gronwall1919note,howard_gronwall_1998}. 

\end{proof}

\section{Application to long-time stability}\label{sec:application_to_long-time_stability}

The existence of the modified energy guarantees the stability of the numerical scheme. The previous result then implies the existence and stability {of numerical solitons}, as in \cite{bambusi2013existence}, as well as normal form result as in \cite{FGP1,FGP2,AbouKhalil2024} yielding preservation of the Fourier modes of the solution over long times for small initial data. 

As a simple example of an application of Theorem~\ref{thm:modified_energy_midpoint}, we provide here an almost global existence result form small initial date in $V_\deltax$. It is based on the idea that for small data, the energy \eqref{energy} controls the norm $\Norm{\cdot}{\deltax}$ for small initial data. This is a consequence of the following result: 

\begin{lemma}
Assume $\deltax$ and $h$ satisfy the condition $h \leq C \deltax^2$. Then we have 
$$
A_0(u) = \bar u^T  \frac{2}{h} \arctan \left( \frac{h A}{2}\right) u \geq c\, \bar u^T A u,
$$
where the constant $c$ only depends on $C$. 
\end{lemma}
\begin{proof}
We have the existence of $c$ such that for all $x \leq C$, 
$$
\arctan(x) \geq c x. 
$$
Then we know that there exists a unitary matrix $U$ such that
$$
A = U^{-1}D U \quad \mbox{with} \quad 
D = \mathrm{diag}(\lambda_j)
$$
where the $\lambda_j$ are given by \eqref{diagonal}. Note that we thus have $\frac12 h \lambda_j \leq C$ by assumption. 
Then we have, with $v = U u$, 
\begin{align*}
A_0(u) &=  \overline{v}^T \frac{2}{h} \arctan \left( \frac{h D}{2}\right) v \\
&= \sum_{j} |v_j|^2 \frac{2}{h} \arctan \left( \frac{h \lambda_j }{2}\right)\\
&\geq \sum_{j} c \lambda_j |v_j|^2 =  \bar u^T A u, \end{align*}
and this shows the result. 
\end{proof}

\begin{theorem}
\label{Theo:global}
Let $u^{n+1} = \varphi_h(u^n)$ be the map defined by the midpoint rule \eqref{eqn:midpoint_rule_nlse}. Fix $N \in \N$ and $\kappa \in (0,\frac12)$.  Then there exists  $\epsilon_0$ and $h_0$ such that for 
$h \leq h_0$ and all $K$ and $\deltax$ satisfying the CFL condition \eqref{eqn:CFL_condition_modified_energy} and all $\epsilon \leq \epsilon_0$, we have 
$$
\Norm{u^0}{\deltax} = \epsilon \quad \Longrightarrow\quad 
\Norm{u^n}{\deltax} \leq \epsilon^{1 - \kappa}, \quad \mbox{for}\quad nh \leq  ( h \epsilon^{2r(1-\kappa)})^{-N},
$$
i.e. the numerical solution is stable in $\|\cdot\|_{\deltax}$ for long times.
\end{theorem}
\begin{proof}
Let $H_h$ be the energy \eqref{energy} given by Theorem \ref{thm:modified_energy_midpoint}. 
Assume that $u^n \in B_{\deltax}(\epsilon^{1 - \kappa})$. Then we have using \eqref{mainest} and the fact that that $H_h(\Phi^h_{H_h}(u)) = H_h(u)$ for all $u$, 
\begin{align*}
|H_h(u^{n+1}) - H_h(u^{n})|&\leq |H_h(\varphi_h(u^n)) - H_h(\Phi^h_{H_h}(u^n))|+|H_h(\Phi^h_{H_h}(u^n))  - H_h(u^{n})|\\
& \leq C_N h^{N+2}  \epsilon^{(2r (N+2)  +  1)(1 - \kappa) }. 
\end{align*}
This shows that as long as $u^n \in B_{\deltax}(\epsilon^{1-\kappa})$, we have 
$$
|H_h(u^{n+1}) - H_h(u^{0})| \leq (nh) C_N h^{N+1} \epsilon^{(2r (N+2)  +  1)(1- \kappa)}. 
$$
But using \eqref{decoB}, we have 
\begin{align*}
|A_0(u^{n}) - A_0(u^{0})|  &\leq |H_h(u^{n+1}) - H_h(u^{0})| + C_N \sum_{k = 0}^N  \epsilon^{(2r(k+1) + 2)(1 - \kappa)} \\
&\leq C_N \epsilon^{(2r + 2)(1-\kappa)} \left[ 1 + (nh) h^{N+1} \epsilon^{2r (1-\kappa)N} \right] 
\leq 2 C_N \epsilon^{(2r + 2)(1-\kappa)} 
\end{align*}
for $(nh) h^{N+1} \epsilon^{2r N(1 - \kappa)} \leq 1$. 
As the $L^2$-norm $N(u)$ is preserved (see \eqref{eq:discrL2}), we deduce that as long as $u^n \in B_{\deltax}(\epsilon^{1-\kappa})$ and for 
$nh \leq ( h\varepsilon^{2r(1-\kappa)})^{-N}$ we have, using the previous Lemma, 
\begin{align*}
\Norm{u^\delta}{\deltax}^2 &\leq \frac{1}{c} \Big( A_0(u^n) + N(u^n) \Big) 
\leq C_N \left(  A_0(u^0) + N(u^0) + \epsilon^{(2r + 2)(1 - \kappa))} \right)
\\ & \leq C_N \left(  \Norm{u^0}{\deltax}^2 + \epsilon^{(2r + 2)(1 - \kappa))} \right)
\leq C_N  \epsilon^2 \leq \epsilon^{2(1-\kappa)}
\end{align*}
as $(r+1)(1-\kappa) > 1$, and 
for $\epsilon$ small enough, and with possible changes of the constant $C_N$ between the lines of the previous estimates. This shows the desired result.  
\end{proof}


\bibliographystyle{plain} 
\bibliography{biblio}

\end{document}